\documentclass[reqno, 12pt]{amsart}
\pdfoutput=1
\makeatletter
\let\origsection=\section \def\section{\@ifstar{\origsection*}{\mysection}}
\def\mysection{\@startsection{section}{1}\z@{.7\linespacing\@plus\linespacing}{.5\linespacing}{\normalfont\scshape\centering\S}}
\makeatother

\usepackage{amsmath,amssymb,amsthm}
\usepackage{mathrsfs}
\usepackage{mathabx}\changenotsign
\usepackage{dsfont}

\usepackage{xcolor}
\usepackage[backref]{hyperref}
\hypersetup{
    colorlinks,
    linkcolor={red!60!black},
    citecolor={green!60!black},
    urlcolor={blue!60!black}
}

\usepackage{graphicx}

\usepackage[open,openlevel=2,atend]{bookmark}

\usepackage[abbrev,msc-links,backrefs]{amsrefs}
\usepackage{doi}

\renewcommand{\PrintDOI}[1]{\doi{#1}}

\usepackage[T1]{fontenc}
\usepackage{lmodern}
\usepackage[babel]{microtype}
\usepackage[english]{babel}

\usepackage{geometry}
\numberwithin{equation}{section}
\numberwithin{figure}{section}

\usepackage{enumitem}
\def\rmlabel{\upshape({\itshape \roman*\,})}

\let\polishlcross=\l
\def\l{\ifmmode\ell\else\polishlcross\fi}

\def\paragraph#1{%
  \noindent\textbf{#1.}\enspace}

\let\emptyset=\varnothing
\let\setminus=\smallsetminus

\let\sm=\setminus

\makeatletter
\def\moverlay{\mathpalette\mov@rlay}
\def\mov@rlay#1#2{\leavevmode\vtop{   \baselineskip\z@skip \lineskiplimit-\maxdimen
   \ialign{\hfil$\m@th#1##$\hfil\cr#2\crcr}}}
\newcommand{\charfusion}[3][\mathord]{
    #1{\ifx#1\mathop\vphantom{#2}\fi
        \mathpalette\mov@rlay{#2\cr#3}
      }
    \ifx#1\mathop\expandafter\displaylimits\fi}
\makeatother

\newcommand{\dcup}{\charfusion[\mathbin]{\cup}{\cdot}}

\DeclareFontFamily{U}  {MnSymbolC}{}
\DeclareSymbolFont{MnSyC}         {U}  {MnSymbolC}{m}{n}
\DeclareFontShape{U}{MnSymbolC}{m}{n}{
    <-6>  MnSymbolC5
   <6-7>  MnSymbolC6
   <7-8>  MnSymbolC7
   <8-9>  MnSymbolC8
   <9-10> MnSymbolC9
  <10-12> MnSymbolC10
  <12->   MnSymbolC12}{}
\DeclareMathSymbol{\powerset}{\mathord}{MnSyC}{180}

\usepackage{tikz}
\usetikzlibrary{calc,decorations.pathmorphing}
\pgfdeclarelayer{background}
\pgfdeclarelayer{foreground}
\pgfdeclarelayer{front}
\pgfsetlayers{background,main,foreground,front}

\let\epsilon=\varepsilon
\let\eps=\epsilon
\let\rho=\varrho
\let\theta=\vartheta
\let\kappa=\varkappa

\let\E=\EE
\def\NN{{\mathds N}}

\def\PP{{\mathds P}}

\def\Var{{\mathds V}ar}

\newcommand{\cA}{\mathcal{A}}
\newcommand{\cB}{\mathcal{B}}
\newcommand{\cC}{\mathcal{C}}
\newcommand{\cE}{\mathcal{E}}
\newcommand{\cF}{\mathcal{F}}
\newcommand{\cG}{\mathcal{G}}
\newcommand{\cH}{\mathcal{H}}

\newcommand{\cK}{\mathcal{K}}

\newcommand{\cP}{\mathcal{P}}

\newcommand{\cX}{\mathcal{X}}

\newcommand{\ccJ}{\mathscr{J}}

\newcommand{\tB}{\tilde{B}}

\theoremstyle{plain}
\newtheorem{thm}{Theorem}[section]
\newtheorem{theorem}[thm]{Theorem}

\newtheorem{prop}[thm]{Proposition}
\newtheorem{claim}[thm]{Claim}
\newtheorem{fact}[thm]{Fact}
\newtheorem{cor}[thm]{Corollary}
\newtheorem{lemma}[thm]{Lemma}

\theoremstyle{definition}
\newtheorem{rem}[thm]{Remark}
\newtheorem{dfn}[thm]{Definition}

\usepackage{accents}
\newcommand{\seq}[1]{\accentset{\rightharpoonup}{#1}}

\let\phi=\varphi

\begin{document}

\title{High powers of Hamiltonian cycles in randomly augmented graphs}

\author[S. Antoniuk]{Sylwia Antoniuk}
\address{Department of Discrete Mathematics, Adam Mickiewicz University, Pozna\'n, Poland}
\email{antoniuk@amu.edu.pl}

\author[A. Dudek]{Andrzej Dudek}
\address{Department of Mathematics, Western Michigan University, Kalamazoo, MI, USA}
\email{andrzej.dudek@wmich.edu}
\thanks{The second author was supported in part by Simons Foundation Grant \#522400.}

\author[Chr. Reiher]{Christian Reiher}
\address{Fachbereich Mathematik, Universit\"at Hamburg, Hamburg, Germany}
\email{christian.reiher@uni-hamburg.de}

\author[A. Ruci\'nski]{Andrzej Ruci\'nski}
\address{Department of Discrete Mathematics, Adam Mickiewicz University, Pozna\'n, Poland}
\email{\tt rucinski@amu.edu.pl}
\thanks{The fourth author was supported in part by the Polish NSC grant 2018/29/B/ST1/00426}

\author[M. Schacht]{Mathias Schacht}
\address{Fachbereich Mathematik, Universit\"at Hamburg, Hamburg, Germany}
\email{schacht@math.uni-hamburg.de}
\thanks{The fifth author was supported by the European Research Council (PEPCo 724903).}

\begin{abstract}
We investigate the existence of powers of Hamiltonian cycles in graphs with large minimum degree to which some additional edges have been added in a random manner. For all integers $k\geq1$,  $r\geq 0$, and $\ell\geq (r+1)r$, and for any $\alpha>\frac{k}{k+1}$ we show that adding $O(n^{2-2/\ell})$ random edges to an $n$-vertex graph $G$ with minimum degree at least $\alpha n$ yields, with probability close to one, the existence of the $(k\ell+r)$-th power of a Hamiltonian cycle. In particular, for $r=1$ and $\ell=2$ this implies that adding $O(n)$ random edges to such a graph $G$ already ensures the $(2k+1)$-st power of a Hamiltonian cycle (proved independently by Nenadov and Truji\'c). In this instance and for several other choices of $k$, $\ell$, and $r$ we can show that our result is asymptotically optimal.
\end{abstract}

\maketitle

\setcounter{footnote}{1}

\section{Introduction}
All graphs we consider are finite. For $m\in\NN$ the \emph{$m$-th power~$F^m$} of a~graph~$F$ is defined
as the graph on the same vertex set whose edges join distinct vertices at distance at most~$m$ in~$F$.
A Hamiltonian cycle in a graph $G$ is a cycle which passes through all vertices of~$G$.
The $m$-th power of an $s$-vertex path $P_s$ or cycle $C_s$ will be often called \emph{the $m$-path} or, respectively, \emph{the $m$-cycle}.

For integers $m\ge1$ and $n\geq m+2$, let us consider the set $\cC_n^m$
of all $n$-vertex graphs~$G$ that contain the $m$-th power $C^m_n$ of a Hamiltonian cycle $C_n$. Clearly,
$\cC_n^m$ is a monotone graph property, as powers of  Hamiltonian cycles cannot disappear
as a result of adding more edges (without new vertices).

The classical theorem of Dirac~\cite{D1952} asserts that every graph of order $n\ge 3$ and minimum degree $\delta(G)\ge n/2$ is Hamiltonian.
Furthermore, the resolution of the P\'osa--Seymour conjecture~\cites{E1964,S1974} (for large $n$), proved by
Koml\'os, Sark\"ozy, and Szemer\'edi~\cite{KSS1998}, yields the following extension: for each $k\ge2$ every $n$-vertex graph $G$ with $n\ge n_0$ and $\delta(G)\ge\tfrac{k}{k+1}n$ possesses property $\cC^{k}_n$.

Suppose that a graph $G$ has large minimum degree which, however, falls short from the above bound. Would adding a few random edges to $G$ help to create the desired power of a Hamiltonian cycle? Bohman, Frieze, and Martin \cite{BFM2003} were the first to study this question. They showed that, for any $\eps>0$, randomly sprinkling $O(n)$ additional edges onto a graph $G$ with $\delta(G)\ge \eps n$ forces,
with high probability, a Hamiltonian cycle~$C_n$. This result was extended in \cite{DRRS} to all $k\ge1$: \emph{for every graph $G$ with $n\ge n_0$ and $\delta(G)\ge (\tfrac k{k+1}+\eps)n$ adding $O(n)$ random edges yields, with high probability, the $(k+1)$-st power of a Hamiltonian cycle~$C^{k+1}_n$.} Note that the result in~\cite{KSS1998} guarantees only the existence of $C_n^k$ in $G$.
In this paper we substantially generalize the result in \cite{DRRS}.

We investigate the probability that a given $n$-vertex graph $G$ with  minimum degree  high enough to yield, by  the P\'osa--Seymour conjecture, $C_n^k$ in $G$,
augmented by a binomial random graph $G(n,p)$, $p=p(n)$, spans the
$m$-th power of a Hamiltonian cycle $C_n^m$. In other words we are interested in~$\PP(G\cup G(n,p)\in\cC_n^m)$.
To this end, we introduce the following definition.

 \begin{dfn}\label{dirac_thr}
  Given integers $k\ge 0$ and $m\ge 1$, we say that a sequence $d(n)$ is a \emph{ $(k,m)$-Dirac threshold} if

 (a) for every $\alpha>\tfrac{k}{k+1}$ there is $C>0$ such that for all $p(n)\ge C d(n)$
 \[
	\lim_{n\to\infty}\min_G\PP\big(G\cup G(n,p(n))\in\cC_n^m\big)=1\,,
\]
where the minimum is taken over all $n$-vertex graphs~$G$ with $\delta(G)\geq \alpha n$, while

(b) there exists $\alpha_0>\tfrac k{k+1}$ such that for all $\tfrac{k}{k+1}<\alpha<\alpha_0$ there is $c>0$ and  a sequence of $n$-vertex graphs $G_\alpha=G_\alpha(n)$ such that $\delta(G_\alpha)\geq \alpha n$ and for all $p\le cd(n)$
\[
	\lim_{n\to\infty}\PP\big(G_\alpha\cup G(n,p(n))\in\cC_n^m\big)=0\,.
\]
We denote  any function $d(n)$ satisfying both conditions (a) and (b) by $d_{k,m}(n)$.
\end{dfn}
In view of this definition, for any $d(n)$ satisfying  condition (a) alone we have $d_{k,m}(n)\le d(n)$, while  for any $d(n)$ satisfying condition (b) alone we have $d_{k,m}(n)\ge d(n)$.
Also $d_{k,m+1}(n)\ge d_{k,m}(n)$, as $C_n^m\subset C_n^{m+1}$.

A careful reader will notice that the above definition assumes that the dependence on $\alpha$ in the threshold $d_{k,m}$ appears only in a multiplicative constant. This is sufficient for our main result, however, there are instances of $k$ and $m$ for which this is not the case (see Section~\ref{sec:con_rem}).

The result in~\cite{BFM2003} can be now restated as $d_{0,1} = n^{-1}$.
Our main result establishes an upper bound on the Dirac threshold  $d_{k,m}(n)$ for infinitely many values of $m$ for each $k\in\NN$.

\begin{theorem}\label{thm:main}
For all integers $k\geq 1$, $r\geq 0$, $\ell\geq r(r+1)$, and for $m=k\ell+r$
\[
d_{k,m}(n)\leq n^{-2/\l}.
\]

\end{theorem}

Notice that the largest $r=r(\ell)$ for which $\ell\geq r(r+1)$ is $r = \lfloor \frac{\sqrt{4\ell+1} -1}{2} \rfloor$. Thus, Theorem~\ref{thm:main} implies that $d_{k,m}(n)\leq n^{-2/\l}$ for any $m\le k\ell  + \frac{\sqrt{4\ell+1} -1}{2}$. Furthermore, for many choices of $k$ and $m$  
we can provide a matching lower bound on $d_{k,m}(n)$, thus,  determining the $(k,m)$-Dirac threshold altogether.

\begin{theorem}\label{thm:main_bounds}
For all positive integers $k$, $\ell$, and $m$ satisfying the inequalities
\[
(k+1)(\ell-1) \le m \le k\ell  + \frac{\sqrt{4\ell+1} -1}{2},
\]
 we have
 \[
 d_{k,m}(n)=n^{-2/\ell}.
 \]
\end{theorem}

In particular, for each $\ell=2,3,4,5,6$ and infinitely many $k$, we list all values of $m$ for which $d_{k,m}(n)$ has been determined in Theorem \ref{thm:main_bounds}.

\begin{cor}\label{cor:main_bounds}
The following holds:
\begin{enumerate}[label=\rmlabel]
\item For $k\ge 1$ and\; $k+1\le m\le 2k+1$ we have $d_{k,m}(n)=n^{-1}$.
\item For $k\ge 1$ and\; $2k+2\le m\le 3k+1$ we have  $d_{k,m}(n)=n^{-2/3}$.
\item For $k\ge 2$ and\; $3k+3\le m\le 4k+1$ we have  $d_{k,m}(n)=n^{-1/2}$.
\item For $k\ge 3$ and\; $4k+4\le m\le 5k+1$ we have   $d_{k,m}(n)=n^{-2/5}$.
\item For $k\ge 3$ and\; $5k+5\le m\le 6k+2$ we have   $d_{k,m}(n)=n^{-1/3}$.
\end{enumerate}
\end{cor}
\noindent
Part (i) of Corollary \ref{cor:main_bounds} was independently proved by Nenadov and Truji\'c in \cite{NT}.

We  also show that the threshold from the last case of Corollary~\ref{cor:main_bounds} can be extended to $k\in \{1,2\}$.

\begin{theorem}\label{thm:6k2}
For every  integer $k\ge 1$,
\[
d_{k,6k+2}(n)=n^{-1/3}.
\]
\end{theorem}

Observe that in view of Corollary~\ref{cor:main_bounds} the first open case is $k=1$ and $m=5$. We will comment on this in Section~\ref{sec:con_rem}.

\section{Random graphs}\label{rg}

There are two basic models of random graphs, the binomial one, $G(n,p)$, and the uniform one, $G(n,M)$, which are asymptotically equivalent under some mild assumptions whenever $M\sim\binom n2p$ (see Section~1.4 in \cite{JLR}). In this paper we chose to state and prove our results in the binomial model, yet they can be translated to the uniform model if there is such a need. For instance, Theorem~\ref{thm:main} asserts that under the assumptions given there it suffices to add $O(n^{2-2/\ell})$ random edges to ensure a copy of $C_n^m$.

In this section we collect  some results on $G(n,p)$ which we use later.
For a graph property~$\cP$, we say that $\cP$ holds \emph{asymptotically almost surely} (\emph{a.a.s.}) if
$\PP(G(n,p)\in\cP)\to1$ as $n\to\infty$.
In our proofs we are going to use the following consequence of Chebyshev's inequality (see Remark~3.7. in \cite{JLR}).

\begin{fact}\label{cheb} For every $\ell\ge3$ and  $\gamma>0$, there is a constant $c=c(\ell,\gamma)$ such that if $p\le cn^{-2/\ell}$, then a.a.s.\ there are fewer than $\gamma n$ copies of the clique $K_\ell$ in $G(n,p)$.
\end{fact}

We will also apply two versions of Janson's inequality.
The most general one is given in  \cite[Theorem~2.14]{JLR}.
For the proof of Theorem~\ref{thm:main} we will need a strengthening of Theorem~3.9 in \cite{JLR} (the R-H-S inequality only), which is a version of Theorem~2.14 in the context of subgraphs of random graphs.
For a graph $G$ with at least one edge, set
\[
\Phi_G=\min_{F\subseteq G, e_F>0} \Psi_F,
\]
where $\Psi_F=n^{v_F}p^{e_F}$, and $v_F$, $e_F$ denote, respectively, the number of vertices and the number of edges of graph $F$.

\begin{theorem}\label{Janson2}
Let $\tau>0$ and $G$ be a graph with at least one edge and let $\cG$ be a family of copies of $G$ in $K_n$ with $|\cG|\ge\tau n^{v_G}$. Further, let $X$ be the number of copies of $G$ belonging to $\cG$ which are present in $G(n,p)$. Then,
\[
\PP(X \le\tau\Psi_G/2)\le\exp\{-\tau^2 4^{-e_G}\Phi_G/8\}.
\]
\end{theorem}

\proof
We follow the lines of the proof of Theorem~3.9 in \cite{JLR}. The main difference is that we rely on Theorem~2.14 from~\cite{JLR}, and not on Theorem~2.18(ii).
Moreover, instead of defining the indicators $I_{G'}$ for all copies of $G$ in $K_n$, we define them for $G'\in\cG$ only.

We have
\[
\lambda:=\E X=|\cG|p^{e_G}\ge\tau n^{v_G}p^{e_G}=\tau\Psi_G.
\]
By Theorem~2.14 from \cite{JLR},
\begin{equation}\label{2.14}
\PP(X\le\tau\Psi_G/2)\le\PP(X\le\lambda/2)\le\exp\{-\lambda^2/(8\bar\Delta)\},
\end{equation}
where $\bar\Delta$ is defined in Theorem~2.14 \cite{JLR}. In our case,
\[
\bar\Delta=\sum_H\sum_{G'}\sum_{G''}p^{2e_{G}-e_H},
\]
where $\sum_H$ is taken over all subgraphs $H$ of $G$ with $e_H>0$, $\sum_{G'}$ -- over all copies $G'$ of $G$ in $\cG$, while $\sum_{G''}$ -- over all copies $G''$ of $G$ in $\cG$ with $G'\cap G''=H$. This  can be upper bounded, estimating crudely  the number of copies of $H$ in $G$ by $2^{e_G}$, as follows:
\[
\sum_H n^{v_G}2^{e_G} n^{v_{G}-v_H}p^{2e_{G}-e_H}\le\sum_H 2^{e_G}\frac{\Psi_G^2}{\Psi_H}\le 4^{e_G}\frac{\Psi_G^2}{\Phi_G}.
\]
Plugging this bound into (\ref{2.14}) completes the proof.
\qed

\section{Lower bounds}
Here we deduce Theorems~\ref{thm:main_bounds} and \ref{thm:6k2} from Theorem~\ref{thm:main} by complementing it with lower bounds on the corresponding $(k,m)$-Dirac thresholds.

\begin{proof}[Proof of Theorem~\ref{thm:main_bounds}] In view of Theorem~\ref{thm:main}, it suffices to show that if $(k+1)(\ell-1) \le m$, then $d_{k,m}(n)\ge n^{-2/\ell}$.

By monotonicity of $d_{k,m}(n)$ as a function of $m$, we may assume that $m=(k+1)(\ell-1)$. Set $\eps_0=(2(m+1)(k+1))^{-1}$ and fix $\alpha=\tfrac k{k+1}+\eps$ for some $\eps\le \eps_0$. Consider the following construction of a graph $G_\alpha$.
For $n\ge 4(k+2)(m+1)$ let  $[n]=V_1\dcup\dots\dcup V_{k+1}$ be a vertex partition
with each part of size~$n/(k+1)$ (for simplicity, we assume that $n$ is divisible by $k+1$).
Moreover, for every $i=1,\dots,k+1$, fix a subset $W_i\subseteq V_i$ of
size~$|W_i|=\lceil \eps n\rceil$. Let $G=G_\alpha$ be the $n$-vertex graph consisting of the union of the complete $(k+1)$-partite graph with vertex partition~$V_1\dcup \dots\dcup V_{k+1}$ and $k+1$ complete bipartite graphs with vertex classes~$W_i$ and~$V_i\setminus W_i$ for $i=1,\dots,k+1$.
Clearly, $\delta(G)\ge(\tfrac{k}{k+1}+\eps)n$. Set $W=\bigcup_{i=1}^{k+1} W_i$, for convenience.

Let $p\le cn^{-2/\l}$, where $c=c(\ell,\gamma)$ is defined in Fact~\ref{cheb} with $\gamma=(4(m+1))^{-1}$. We are going to show that a.a.s. $H=G\cup G(n,p)$ does not contain any copy of~$C^{m}_n$.
Note that any $C^m_n\subset H$
contains $\lfloor n/(m+1)\rfloor$ vertex-disjoint copies of~$K_{m+1}$ and only at most $\big|W\big|=(k+1)\lceil \eps n\rceil$ of them have a vertex in $W$.

Consider a copy $K$ of~$K_{m+1}$ which is disjoint from $W$.
Since $\lceil \frac{m+1}{k+1}\rceil = \ell$, by Pigeonhole Principle, $K$ must contain a copy $K'$ of $K_\ell$ that lies entirely in some set~$V_i$ and thus $K'$ must be a subgraph of $G(n,p)$. In conclusion, if $C^{m}_n \subset H$, then the random graph $G(n,p)$ must contain at least
\[
\left\lfloor \frac{n}{m+1}\right\rfloor-(k+1)\lceil \eps n\rceil
\ge \frac{n}{m+1}-1 -(k+1)(\eps_0 n+1)
= \frac{n}{2(m+1)}-k-2
\ge \gamma n
\]
copies of $K_\l$. However, by Fact~\ref{cheb}, for $p\le c n^{-2/\l}$, a.a.s.\ there are fewer than $\gamma n$ copies of $K_\l$ in $G(n,p)$. This establishes part (b) of Definition \ref{dirac_thr} with $\alpha_0=(\tfrac k{k+1}+\eps_0)$.
\end{proof}

\medskip


\medskip

\begin{proof}[Proof of Theorem~\ref{thm:6k2}]
Let $m=6k+2$ and $k=1,2$. (For $k\ge3$, Theorem~\ref{thm:6k2} follows from Corollary~\ref{cor:main_bounds}(v).) Theorem~\ref{thm:main}, applied with $\ell=6$ and $r=2$, yields that $d_{k,m}\le n^{-1/3}$.
We will show that also $d_{k,m}\ge n^{-1/3}$. For $\eps_1=1/288$ and $\eps_2=1/105$, fix $\alpha_k=\tfrac k{k+1}+\eps_k$, $k=1,2$, and consider the graph $G_{\alpha_k}$ constructed in the proof of Theorem~\ref{thm:main_bounds}. Further, take $p\le cn^{-1/3}$, where $c=c(6,\eps_k)$ is defined in Fact~\ref{cheb}. Consequently, there are a.a.s.\ no more than $\eps_k n$ copies of $K_6$ in $G(n,p)$.

Assume that $H=G \cup G(n,p)$ contains a copy $C$ of $C_n^{m}$. After removing from $H$ all vertices in~$W$ as well as at least one vertex from each copy of $K_6$ in $G(n,p)$, we obtain a $K_6$-free subgraph $H' \subset H$ on $n'\ge n-(k+2)\eps_k n$ vertices such that $H'[V_j]\subset G(n,p)$ for $1\le j\le k+1$. Observe that $H' \cap C$ contains an $m$-path~$P$ with
\[
|V(P)|\ge \frac{n'}{(k+2)\eps_k n} \ge \frac{1}{(k+2)\eps_k}-1\ge \frac{1}{(k+3)\eps_k}.
\]
Now, consider separately the cases $k=2$ and $k=1$. The former one is a bit easier. We  have $m=14$ and $H'$ is tripartite.
For each $1\le j\le 3$, the subgraph  $Q_j=P[V_j]$ contains a spanning 4-path in $G(n,p)$. Indeed, let $v_1,\dots,v_5$ be five consecutive vertices  of $Q_j$ (in the linear order determined by $P$). Since there are no copies of $K_6$ in $H'$, there are at most $3+2\cdot 5$ vertices between $v_1$ and $v_5$ on $P$. Therefore, $v_s$ and $v_t$, $1\le s<t\le 5$, are adjacent in $P$ and thus in $Q_j$.

Let $q=\max\{|V(Q_1)|,|V(Q_2)|,|V(Q_3)|\}$. Then, $G(n,p)$ contains a 4-path with
\[
q\ge \frac{1}{3}|V(P)| \ge \frac{1}{15\eps_2} = 7
\]
vertices and $4q-(1+2+3) = 4q-6$ edges. However, the expected number of such subgraphs in $G(n,p)$ with $p = O(n^{-1/3})$  is smaller than
\[
n^qp^{4q-6}=O(n^{-q/3+2}) = O(n^{-7/3+2}) =o(1),
\]
which by Markov's inequality implies that a.a.s.\ there are not such copies at all.

For $k=1$, we have $m=8$ and $H'$ is bipartite. Now, we can only claim that each $Q_j$ contains a spanning 3-path $P_j$ in $G(n,p)$. Indeed, let $v_1,\dots,v_4$ be four consecutive vertices  of $Q_j$. Similarly to the case $k=2$, $v_1,\dots, v_4$ form a clique in $G(n,p)$. Fortunately, there are more edges in $Q_j$. To see it,
divide  $V(P)$ into $\lfloor |V(P)|/9 \rfloor$ consecutive copies of $K_9$. Each of them contributes a copy of $K_5$ to either $Q_1$ or $Q_2$ and thus an extra edge which is not present in $P_j$. Without loss of generality, suppose that $Q_1$ contains at least $\tfrac12\lfloor|V(P)|/9\rfloor$ copies of~$K_5$. Then $Q_1$ is a subgraph of $G(n,p)$ with $q$ vertices and at least
\[
|P_1| + \frac{|V(P)|}{18} - \frac{1}{2} \ge
3q-3+\frac{1}{72\eps_1} - \frac{1}{2} = 3q+\frac{1}{2}
\] edges. Again, with $p = O(n^{-1/3})$, by Markov's inequality, there are no such subgraphs in $G(n,p)$.
\end{proof}

\section{Outline of the proof of Theorem \ref{thm:main}}

The proof of Theorem~\ref{thm:main} is based on the {\it absorption method} and
follows  the general outline of the proof in \cite{DRRS}.
It relies on four lemmas, the Connecting Lemma, the Reservoir Lemma, the Absorbing Lemma, and the Covering Lemma. The last three of these lemmas will be stated here and proved in the forthcoming sections. At the end of this section we  provide a short proof of Theorem \ref{thm:main} based on these  lemmas. The Connecting Lemma is used only in the proof of Absorbing Lemma and will be stated and proved in Section \ref{sec:connect}.
Each of these lemmas provides the existence of some $m$-paths in $H=G\cup G(n,p)$, so the proofs  involve mixed techniques from extremal graph theory and random graphs.

Throughout the rest of the paper we assume that $m = k\l+r$ and $\l\ge r(r+1)$, where $k,\ell\in \NN$ and $r\ge 0$.
Observe that if $\l=1$, then necessarily $r=0$ and so $m=k$. This case, however, follows deterministically from~\cite{KSS1998}, since $\delta(G)\ge kn/(k+1)$.
Therefore, from now on we will be assuming that $\ell \ge 2$.  Note that by the monotonicity of $d_{k,m}(n)$ as a function of $m$, it is enough to consider only the largest $r$ satisfying $\ell\ge r(r+1)$. In particular, in view of $\ell \ge 2$, we may also assume that $r\ge 1$.

Given an $m$-path $P=(v_1,\dots,v_t)$, the sequences $(v_1,\dots,v_{m})$ and $(v_t,\dots,v_{t-m+1})$ are called the \emph{ends} of $P$.
We say that $P$ \emph{connects} its ends and the vertices of $P$ not belonging to its ends are called \emph{internal}.
As every segment of consecutive $m+1$ vertices of an $m$-path forms a clique $K_{m+1}$, the ends span $m$-cliques. If $K$ and $K'$ are the ordered cliques induced by the
ends of an $m$-path~$P$, we may also say that~$P$ \emph{connects} $K$ and $K'$ .

\begin{dfn}\label{connectable_tuples}
Given $\xi>0$, an $m$-tuple $\seq{x} = (x_1, x_2,\ldots,x_m)$ of vertices of an $n$-vertex graph $G$ is \textit{$\xi$-connectable} if there exist $\xi n^{k+1}$  (ordered) copies $(y_1, y_2,\ldots,y_{k+1})$ of $K_{k+1}$ in $G$ with the property that for each $i = 1, 2,\ldots,k+1$, $y_i\in N_G(x_{(i-1)\ell+1},\ldots,x_m)$.
An $m$-path in $H$ is \emph{$\xi$-connectable} if both its ends are $\xi$-connectable $m$-tuples in $G$.
\end{dfn}
\noindent Note that for $0<\xi_1<\xi_2$, if an $m$-tuple is $\xi_2$-connectable, then is is also $\xi_1$-connectable.

We may now state the  Reservoir Lemma which is proved in Section~\ref{sec:connect}. Here and below $V$ always stands for the vertex set of the graphs $G$, $G(n,p)$, and $H$.

\begin{lemma}[Reservoir Lemma]\label{prop:reservoir}
For all  $\eps>0$ and $\xi>0$  there exists $\gamma>0$
such that for all sufficiently large $C=C(\xi,\gamma)\ge1$  and for every $n$-vertex graph $G$ with
$\delta(G)\geq (\tfrac{k}{k+1}+\eps)n$ there exists a set of vertices $R\subseteq V$
of size $\tfrac12\gamma^2n\le|R|\le2\gamma^2n$
such that for $p=p(n)\geq  Cn^{-2/\l}$ a.a.s.\ $H=G\cup G(n,p)$ has the following property.

For every $S\subseteq R$ with  $|S|\leq \sqrt\gamma |R|$ and for every pair of disjoint,
ordered $\xi$-connectable $m$-tuples $\seq{x}, \seq{x}'$ in~$G-R$, there exists an~$m$-path in $H$ connecting  $\seq{x}$ and $ \seq{x}'$
with~$\l(k+1)2^{k+1}$ internal vertices, all from $R\setminus S$.
\end{lemma}

The next result (proved in Section~\ref{sec:absorb}) yields the existence of an $m$-path~$A$, called {\it absorbing}, which can absorb any small set of
vertices. This enables us to reduce our goal to an easier problem of finding an \emph{almost} spanning $m$-cycle
containing~$A$. 

\begin{lemma}[Absorbing Lemma]\label{prop:absorbing}
For every  $\eps>0$ there exists $\xi>0$  such that for  sufficiently small $\gamma=\gamma(\eps,\xi)>0$ and sufficiently large $C=C(\eps,\xi)\ge1$, every $n$-vertex graph $G$ with $\delta(G)\geq (\tfrac{k}{k+1}+\eps)n$,
and every $p = p(n)\ge Cn^{-2/\l}$, a.a.s.\ $H=G\cup G(n,p)$ has the following property.

For every set of vertices $R\subseteq V$ with $|R|\le2\gamma^2n$ the graph $H-R$ contains
a $\xi$-connectable $m$-path $A$ with $|V(A)|\le\gamma n/2$ such that for every $U\subseteq V\setminus V(A)$
with $|U|\le 3\gamma^2n$ there exists an $m$-path $A_U$
	with $V(A_U)=V(A)\cup U$ having the same ends as~$A$.
\end{lemma}

The last lemma below states that almost the whole graph under consideration can be covered by a linear in $n$ number of $m$-paths. 
These  paths will be eventually connected together with the absorbing path $A$,
to produce the desired $m$-th power of an almost spanning cycle. We shall prove the
Covering Lemma in Section~\ref{sec:cover}.

\begin{lemma}[Covering Lemma]\label{prop:covering}
For every  $\eps>0$ there exist $\xi>0$ and $\gamma>0$ such that for sufficiently large $C=C(\xi,\gamma)\ge1$,
for every $n$-vertex graph $G$ with $\delta(G)\geq (\tfrac{k}{k+1}+\eps)n$
and every $p = p(n)\ge Cn^{-2/\l}$, a.a.s.\ $H=G\cup G(n,p)$ has the following property.

For every subset $Q\subset V$ with $|Q|\le\gamma n$ there exists a family~$\cP$ of at most $\gamma^3 n$ vertex-disjoint $\xi$-connectable $m$-paths in $H$ with vertices in $V \setminus Q$ covering all but at most $\gamma^2 n$ vertices of $V\setminus Q$.
\end{lemma}

We conclude the present section with a proof of our main result assuming the three
lemmas stated above. Although the statements of Lemmas \ref{prop:reservoir} - \ref{prop:covering} are not monotone in~$\gamma$, it follows from the three proofs (see Sections 6-8) that whenever they are true for some $\gamma_0>0$, they are also true for any $0<\gamma<\gamma_0$.

\begin{proof}[Proof of Theorem~\ref{thm:main}] 
We begin by fixing the constants. During this process we adopt the convention that a constant coming from Lemma 4.x receives a subscript x.
Let $k\in \NN$ and $\alpha\in\bigl(\frac k{k+1}, 1\bigr)$ be given and set~$\eps=\alpha-\frac k{k+1}$. Plugging $\eps$ into Lemmas~\ref{prop:absorbing} and~\ref{prop:covering} we obtain, respectively, constants $\xi_3$, $\gamma_3$, $C_3$ and  $\xi_4$, $\gamma_4$, $C_4$. Plugging $\xi = \min \{\xi_3,\xi_4\}$ into Lemma~\ref{prop:reservoir} we obtain $\gamma_2$ and $C_2$. Finally, we set
\[
\gamma=\min\left\{\gamma_2, \gamma_3, \gamma_4, \frac14, \left(\l(k+1)2^{k+2}\right)^{-2}\right\}
\]
and $C=\max\{C_2, C_3, C_4\}$.
	
Let an $n$-vertex graph $G$ with $\delta(G)\geq (\tfrac{k}{k+1}+\eps)n$ and $p\ge Cn^{-2/\l}$ be given. We need to check that a.a.s.\ the graph $H=G\cup G(n,p)$ contains a copy of $C_n^{m}$.
For this purpose it suffices to prove that a graph $H$ satisfying the conclusions of all three lemmas above contains a copy of~$C_n^{m}$.

By Lemma~\ref{prop:reservoir} there is a reservoir set $R\subseteq V$ of size $\tfrac12\gamma^2n\le|R|\le2\gamma^2n$. By Lemma~\ref{prop:absorbing} there exists an absorbing $m$-path $A\subseteq H-R$. Since $|R|+|V(A)|\le(2\gamma^2+\gamma/2)n<\gamma n$, we can apply Lemma~\ref{prop:covering} to $Q=R\cup V(A)$ and obtain a collection $\cP$ of at most $\gamma^3n$ vertex-disjoint $\xi$-connectable $m$-paths in $H-Q$
whose vertices cover the set $V \setminus Q$  except for a small subset $U'\subseteq V\sm Q$ with $|U'|\le \gamma^2n$. Next, we want to create the $m$-th power of a long cycle in $H$ by connecting together all paths in $\cP \cup \{A\}$.

To this end, we make $|\cP|+1$ successive applications of Lemma~\ref{prop:reservoir}. In each of them we let $\seq{x}$ and $\seq{x}'$ be the ends of the $m$-paths we wish to connect and let $S\subseteq R$ be the set of  vertices that were used as internal vertices in previous applications. When arriving at the last step of this process, that is, when closing the $m$-cycle, the set $S$ of vertices we need to avoid has size
	\[
		|S|=\ell (k+1)2^{k+1} |\cP|\le \ell(k+1)2^{k+1} \gamma^3 n \le \l(k+1)2^{k+2}\gamma |R| \le
		\sqrt\gamma|R|\,,
	\]
which justifies  repeated applications of Lemma~\ref{prop:reservoir}.
	
Let $F$ be the obtained $m$-cycle. The complement $U=V\sm V(F)$ satisfies
\[
|U|=|U'|+|R\sm V(F)|\le 3\gamma^2n,
\]
whence, by Lemma~\ref{prop:absorbing}, there exists an $m$-path $A_U$ with $V(A_U)=V(A)\cup U$ having the same ends as $A$. Therefore, we can replace $A$ by $A_U$ in $F$ and obtain the desired $m$-th power of a Hamiltonian cycle in $ H$.
\end{proof}

\section{Preliminaries} In this section we present results which serve as tools in the  proofs of the lemmas stated in the previous section.

\subsection{Neighbourhoods in graphs of large minimum degree}\label{sec:degrees}
We recall the following standard notation. For a set $V$ and an integer $j\in\NN$ we write $\binom Vj$ for the family of all $j$-element subsets of $V$. Given a graph $G=(V,E)$ we write $N_G(u)$
for the neighbourhood of a vertex~$u\in V$ in $G$.
More generally, for a subset $U\subseteq V$ we set
\[
	N_G(U)=\bigcap_{u\in U}N(u)
\]
for the \emph{joint neighbourhood} of $U$. For simplicity we may suppress $G$ in the subscript and
for sets $\{u_1,\dots,u_r\}$ we may write $N(u_1,\dots,u_r)$ instead of $N(\{u_1,\dots,u_r\})$. We will use the following result from~\cite[Lemma 3.1]{DRRS}.

\begin{prop}\label{lem:scale}
	For every integer $k\geq 0$ and $\eps>0$ the following holds for every $n$-vertex graph $G=(V,E)$
	with $\delta(G)\geq(\tfrac{k}{k+1}+\eps)n$. For every $j\in[k+1]$ and every $J\in \binom Vj$
 	we have
	\begin{equation}\label{claim1:N}
		|N(J)|\ge\left(\frac{k+1-j}{k+1}+j\eps \right)n\,.
	\end{equation}
	Furthermore, for $j\in[k]$ the induced subgraph $G[N(J)]$ satisfies
	\begin{equation}\label{claim1:d}
		\delta(G[N(J)])
		\ge
		\left(\frac {k-j}{k-j+1}+\eps \right)|N(J)|
	\end{equation}
	for every $J\in \binom Vj$.
	\end{prop}

\subsection{The decomposition} We begin with a crucial decomposition of the $m$-path into two subgraphs.

\begin{dfn}
For $r\ge2$, two sequences of vertices $\seq{v} = (v_1, v_2, \ldots, v_r)$ and $\seq{u} = (u_1, u_2, \ldots, u_r)$ of a graph $G$ are said to be \emph{$r$-bridged} if each $v_i$ is adjacent in $G$ to all $u_1, u_2, \ldots, u_{i}$, $i = 1,2, \ldots, r$, or, equivalently, if each $u_i$ is adjacent in $G$ to all $v_i, v_{i+1}, \ldots, v_r$, $i = 1,2, \ldots, r$.  We then also say that the two sequences form \emph{an $r$-bridge}, or just \emph{a bridge} if the value of $r$ is clear from the context.
\end{dfn}

The first ingredient of the decomposition consists of a number of cliques tied together by bridges to form a linear structure resembling a braid.

\begin{dfn}\label{bra}
For $t\ge1$, $\l\ge2$, and  $1\le r\le\l$, let $B(\l,r,t)$ be \emph{the braid  graph} consisting of $t$ vertex-disjoint $\l$-cliques $K_\l^{(1)},K_\l^{(2)},\ldots,K_\l^{(t)}$, with vertices ordered arbitrarily,  where for each $i=1,\dots,t-1$, the last $r$ vertices of $K_\ell^{(i)}$ and the first $r$ vertices of $K_\ell^{(i+1)}$ are  $r$-bridged. For any $s\geq 1$, we denote by $sB(\l,r,t)$, the union of $s$ vertex disjoint copies of $B(\l,r,t)$.
\end{dfn}
Note that  $B(\l,r,t)$ has $t\l$ vertices and $t\binom\l 2+(t-1)\binom{r+1}2$ edges. Also, for $r\in\{\ell-1,\ell\}$,
$B(\l,r,t)=P^r_{t\ell}$, while $B(\l,1,t)$ consists of $t$ cliques $K_\ell$ connected together by $t-1$ disjoint edges.

The second component of the decomposition involves the notion of the blow-up.

\begin{dfn}\label{blo}
For a graph $F=(V,E)$ with $V=\{v_1,\ldots,v_h\}$, \emph{the $(t_1,\ldots,t_h)$-blow-up of $F$}  is a graph $F(t_1,\dots,t_h)$ obtained from $F$ by replacing each vertex $v_i$ by a set $U_i$ of $t_i$ vertices and each edge $\{v_i,v_j\}$ by the complete bipartite graph $K_{t_i,t_j}$ on $U_i\cup U_j$.
If $t_1=\ldots=t_h=t$ then we call such a graph the $t$-blow-up of $F$ and denote it by $F(t)$.
\end{dfn}
Throughout we will use the convention that if a $k$-path $P^k_h$ has its vertex set listed as a sequence $(v_1,\dots,v_h)$, then we list the vertices of its blow-up $P^k_h(t)$ so that all vertices of $U_1$ precede (in any order) all vertices of $U_2$, which precede all vertices of $U_3$, etc.

We are now ready to describe the decomposition, or, in fact, an embedding of an $m$-path into the union of two edge-disjoint subgraphs.

\begin{prop}\label{m-path decomposition}
Let $k, t \geq 1$,  and $m = k\l + r$, with $1 \leq r \leq \l$. For any copy $P$ of $P^{k}_{(k+1)t}(\l)$, there exists a copy $B$ of $(k+1)B(\l,r,t)$ on $V(P)$, which is edge-disjoint from $P$, and such that one can find a copy of the $m$-path $P_{\ell (k+1)t}^{m}$ in the union of $P$ and $B$, whose vertices inherit the ordering of vertices of $P$, i.e.
\begin{equation}\label{P}
P_{\ell (k+1)t}^{m} \subseteq P^{k}_{(k+1)t}(\l) \dcup (k+1)B(\l,r,t).
\end{equation}
Moreover, for $t$ even and $C = C^{k}_{2k+2}(\l t/2)$, one can find a copy $P$ of $P^{k}_{(k+1)t}(\l)$ in $C$ and then a~copy $B$ of $(k+1)B(\l,r,t)$ on $V(P)$, which is edge-disjoint from $C$ and such that $C \cup B$ contains a~copy of the $m$-path $P_{\ell (k+1)t}^{m}$, whose vertices inherit the ordering of vertices of $P$, i.e.
\begin{equation}\label{C}
P_{\l (k+1)t}^{m} \subseteq C^{k}_{2k+2}(\l t/2) \dcup (k+1)B(\l,r,t).
\end{equation}
\end{prop}

\begin{rem}
For $r\in \{\l-1,\l \}$, the embedding in (\ref{P}) is an actual decomposition, while for $1 \le r \le \l-2$, the embedding omits some of the edges of $P^k_{(k+1)t}(\l)$.
\end{rem}

\begin{rem}\label{1-1}
 Observe that for two paths $P$ and $P'$, there might be the same copy of $B$ satisfying the conditions of Proposition \ref{m-path decomposition}. However, this cannot happen if the paths have different vertex sets.
\end{rem}

\begin{figure}
\centering
\begin{tikzpicture}[scale=0.5]


\coordinate(v1) at (0,0);
\coordinate(v2) at (0.5,0);
\coordinate(v3) at (1,0);

\coordinate(v4) at (2.5,0);
\coordinate(v5) at (3,0);
\coordinate(v6) at (3.5,0);

\coordinate(v7) at (5,0);
\coordinate(v8) at (5.5,0);
\coordinate(v9) at (6,0);

\coordinate(v10) at (7.5,0);
\coordinate(v11) at (8,0);
\coordinate(v12) at (8.5,0);

\coordinate(v13) at (10,0);
\coordinate(v14) at (10.5,0);
\coordinate(v15) at (11,0);

\coordinate(v16) at (12.5,0);
\coordinate(v17) at (13,0);
\coordinate(v18) at (13.5,0);

\coordinate(v19) at (15,0);
\coordinate(v20) at (15.5,0);
\coordinate(v21) at (16,0);

\coordinate(v22) at (17.5,0);
\coordinate(v23) at (18,0);
\coordinate(v24) at (18.5,0);

\coordinate(v25) at (20,0);
\coordinate(v26) at (20.5,0);
\coordinate(v27) at (21,0);

\draw[fill=lightgray] (1,0.2) -- (21,0.2) -- (21,-0.2) -- (1,-0.2) -- (1,0.2);

\node (U1) at (0.5,-1) {$\seq{u}_1$};
\node (U2) at (3,-1) {$\seq{u}_2$};
\node (U3) at (5.5,-1) {$\seq{u}_3$};
\node (U4) at (8,-1) {$\seq{u}_4$};
\node (U5) at (10.5,-1) {$\seq{u}_5$};
\node (U6) at (13,-1) {$\seq{u}_6$};
\node (U7) at (15.5,-1) {$\seq{u}_7$};
\node (U8) at (18,-1) {$\seq{u}_8$};
\node (U9) at (20.5,-1) {$\seq{u}_9$};

\node (U1') at (2,-5.5) {$\seq{u}_1$};
\node (U2') at (4.5,-8.5) {$\seq{u}_2$};
\node (U3') at (7,-11.5) {$\seq{u}_3$};
\node (U4') at (9.5,-5.5) {$\seq{u}_4$};
\node (U5') at (12,-8.5) {$\seq{u}_5$};
\node (U6') at (14.5,-11.5) {$\seq{u}_6$};
\node (U7') at (17,-5.5) {$\seq{u}_7$};
\node (U8') at (19.5,-8.5) {$\seq{u}_8$};
\node (U9') at (22,-11.5) {$\seq{u}_9$};

\begin{scope}[shift={(3,-1.2)}]
    \draw[fill=lightgray] (30:2.6) -- (30:3) arc (30:150:3) -- (150:2.6) arc (150:30:2.6) -- cycle;
\end{scope}

\begin{scope}[shift={(5.5,-1.2)}]
    \draw[fill=lightgray] (30:2.6) -- (30:3) arc (30:150:3) -- (150:2.6) arc (150:30:2.6) -- cycle;
\end{scope}

\begin{scope}[shift={(8,-1.2)}]
    \draw[fill=lightgray] (30:2.6) -- (30:3) arc (30:150:3) -- (150:2.6) arc (150:30:2.6) -- cycle;
\end{scope}

\begin{scope}[shift={(10.5,-1.2)}]
    \draw[fill=lightgray] (30:2.6) -- (30:3) arc (30:150:3) -- (150:2.6) arc (150:30:2.6) -- cycle;
\end{scope}

\begin{scope}[shift={(13,-1.2)}]
    \draw[fill=lightgray] (30:2.6) -- (30:3) arc (30:150:3) -- (150:2.6) arc (150:30:2.6) -- cycle;
\end{scope}

\begin{scope}[shift={(15.5,-1.2)}]
    \draw[fill=lightgray] (30:2.6) -- (30:3) arc (30:150:3) -- (150:2.6) arc (150:30:2.6) -- cycle;
\end{scope}

\begin{scope}[shift={(18,-1.2)}]
    \draw[fill=lightgray] (30:2.6) -- (30:3) arc (30:150:3) -- (150:2.6) arc (150:30:2.6) -- cycle;
\end{scope}

\draw[fill=white] (-0.3,-0.3) -- (1.3,-0.3) -- (1.3,0.3) -- (-0.3,0.3) -- (-0.3,-0.3);

\begin{scope}[shift={(2.5,0)}]
    \draw[fill=white] (-0.3,-0.3) -- (1.3,-0.3) -- (1.3,0.3) -- (-0.3,0.3) -- (-0.3,-0.3);
\end{scope}

\begin{scope}[shift={(5,0)}]
    \draw[fill=white] (-0.3,-0.3) -- (1.3,-0.3) -- (1.3,0.3) -- (-0.3,0.3) -- (-0.3,-0.3);
\end{scope}

\begin{scope}[shift={(7.5,0)}]
    \draw[fill=white] (-0.3,-0.3) -- (1.3,-0.3) -- (1.3,0.3) -- (-0.3,0.3) -- (-0.3,-0.3);
\end{scope}

\begin{scope}[shift={(10,0)}]
    \draw[fill=white] (-0.3,-0.3) -- (1.3,-0.3) -- (1.3,0.3) -- (-0.3,0.3) -- (-0.3,-0.3);
\end{scope}

\begin{scope}[shift={(12.5,0)}]
    \draw[fill=white] (-0.3,-0.3) -- (1.3,-0.3) -- (1.3,0.3) -- (-0.3,0.3) -- (-0.3,-0.3);
\end{scope}

\begin{scope}[shift={(15,0)}]
    \draw[fill=white] (-0.3,-0.3) -- (1.3,-0.3) -- (1.3,0.3) -- (-0.3,0.3) -- (-0.3,-0.3);
\end{scope}

\begin{scope}[shift={(17.5,0)}]
    \draw[fill=white] (-0.3,-0.3) -- (1.3,-0.3) -- (1.3,0.3) -- (-0.3,0.3) -- (-0.3,-0.3);
\end{scope}

\begin{scope}[shift={(20,0)}]
    \draw[fill=white] (-0.3,-0.3) -- (1.3,-0.3) -- (1.3,0.3) -- (-0.3,0.3) -- (-0.3,-0.3);
\end{scope}

\begin{pgfonlayer}{front}
		\foreach \i in {v1,v2,v3,v4,v5,v6,v7,v8,v9,v10,v11,v12,v13,v14,v15,v16,v17,v18,v19,v20,v21,v22,v23,v24,v25,v26,v27}
			 \fill  (\i) circle (3.2pt);		
\end{pgfonlayer}	

\coordinate(w1) at (0,-5);
\coordinate(w2) at (0.5,-5);
\coordinate(w3) at (1,-5);

\coordinate(w4) at (2.5,-8);
\coordinate(w5) at (3,-8);
\coordinate(w6) at (3.5,-8);

\coordinate(w7) at (5,-11);
\coordinate(w8) at (5.5,-11);
\coordinate(w9) at (6,-11);

\coordinate(w10) at (7.5,-5);
\coordinate(w11) at (8,-5);
\coordinate(w12) at (8.5,-5);

\coordinate(w13) at (10,-8);
\coordinate(w14) at (10.5,-8);
\coordinate(w15) at (11,-8);

\coordinate(w16) at (12.5,-11);
\coordinate(w17) at (13,-11);
\coordinate(w18) at (13.5,-11);

\coordinate(w19) at (15,-5);
\coordinate(w20) at (15.5,-5);
\coordinate(w21) at (16,-5);

\coordinate(w22) at (17.5,-8);
\coordinate(w23) at (18,-8);
\coordinate(w24) at (18.5,-8);

\coordinate(w25) at (20,-11);
\coordinate(w26) at (20.5,-11);
\coordinate(w27) at (21,-11);

\begin{pgfonlayer}{front}
		\foreach \i in {w1,w2,w3,w4,w5,w6,w7,w8,w9,w10,w11,w12,w13,w14,w15,w16,w17,w18,w19,w20,w21,w22,w23,w24,w25,w26,w27}
			 \fill  (\i) circle (3.2pt);		
\end{pgfonlayer}

\begin{scope}[shift={(0,-5)}]
\draw[fill=white] (-0.3,-0.3) -- (1.3,-0.3) -- (1.3,0.3) -- (-0.3,0.3) -- (-0.3,-0.3);
\draw (0,0) -- (1,0) edge[out=120,in=60] (0,0);
\begin{scope}[shift={(7.5,0)}]
    \draw[fill=white] (-0.3,-0.3) -- (1.3,-0.3) -- (1.3,0.3) -- (-0.3,0.3) -- (-0.3,-0.3);
    \draw (0,0) -- (1,0) edge[out=120,in=60] (0,0);
    \draw (0.5,0) edge[out=130,in=50] (-6.5,0);
    \draw (-6.5,0) edge[out=40,in=140] (0,0);
    \draw (0,0) edge[out=130,in=50] (-7,0);
\end{scope}
\begin{scope}[shift={(15,0)}]
    \draw[fill=white] (-0.3,-0.3) -- (1.3,-0.3) -- (1.3,0.3) -- (-0.3,0.3) -- (-0.3,-0.3);
    \draw (0,0) -- (1,0) edge[out=120,in=60] (0,0);
    \draw (0.5,0) edge[out=130,in=50] (-6.5,0);
    \draw (-6.5,0) edge[out=40,in=140] (0,0);
    \draw (0,0) edge[out=130,in=50] (-7,0);
\end{scope}
\end{scope}

\begin{scope}[shift={(2.5,-8)}]
\draw[fill=white] (-0.3,-0.3) -- (1.3,-0.3) -- (1.3,0.3) -- (-0.3,0.3) -- (-0.3,-0.3);
\draw (0,0) -- (1,0) edge[out=120,in=60] (0,0);
\begin{scope}[shift={(7.5,0)}]
    \draw[fill=white] (-0.3,-0.3) -- (1.3,-0.3) -- (1.3,0.3) -- (-0.3,0.3) -- (-0.3,-0.3);
    \draw (0,0) -- (1,0) edge[out=120,in=60] (0,0);
    \draw (0.5,0) edge[out=130,in=50] (-6.5,0);
    \draw (-6.5,0) edge[out=40,in=140] (0,0);
    \draw (0,0) edge[out=130,in=50] (-7,0);
\end{scope}
\begin{scope}[shift={(15,0)}]
    \draw[fill=white] (-0.3,-0.3) -- (1.3,-0.3) -- (1.3,0.3) -- (-0.3,0.3) -- (-0.3,-0.3);
    \draw (0,0) -- (1,0) edge[out=120,in=60] (0,0);
    \draw (0.5,0) edge[out=130,in=50] (-6.5,0);
    \draw (-6.5,0) edge[out=40,in=140] (0,0);
    \draw (0,0) edge[out=130,in=50] (-7,0);
\end{scope}
\end{scope}

\begin{scope}[shift={(5,-11)}]
\draw[fill=white] (-0.3,-0.3) -- (1.3,-0.3) -- (1.3,0.3) -- (-0.3,0.3) -- (-0.3,-0.3);
\draw (0,0) -- (1,0) edge[out=120,in=60] (0,0);
\begin{scope}[shift={(7.5,0)}]
    \draw[fill=white] (-0.3,-0.3) -- (1.3,-0.3) -- (1.3,0.3) -- (-0.3,0.3) -- (-0.3,-0.3);
    \draw (0,0) -- (1,0) edge[out=120,in=60] (0,0);
    \draw (0.5,0) edge[out=130,in=50] (-6.5,0);
    \draw (-6.5,0) edge[out=40,in=140] (0,0);
    \draw (0,0) edge[out=130,in=50] (-7,0);
\end{scope}
\begin{scope}[shift={(15,0)}]
    \draw[fill=white] (-0.3,-0.3) -- (1.3,-0.3) -- (1.3,0.3) -- (-0.3,0.3) -- (-0.3,-0.3);
    \draw (0,0) -- (1,0) edge[out=120,in=60] (0,0);
    \draw (0.5,0) edge[out=130,in=50] (-6.5,0);
    \draw (-6.5,0) edge[out=40,in=140] (0,0);
    \draw (0,0) edge[out=130,in=50] (-7,0);
\end{scope}
\end{scope}

\end{tikzpicture}

\caption{For $k=2$, $\l=3$, $r=2$, $t=3$, $m=8$, an $8$-path on 27 vertices $P_{27}^8$ can be embedded into the union of a $3$-blow-up of a $2$-path $P_9^2(3)$ and three copies of the braid graph $B(3,2,3)$.}\label{fig:completion}

\end{figure}
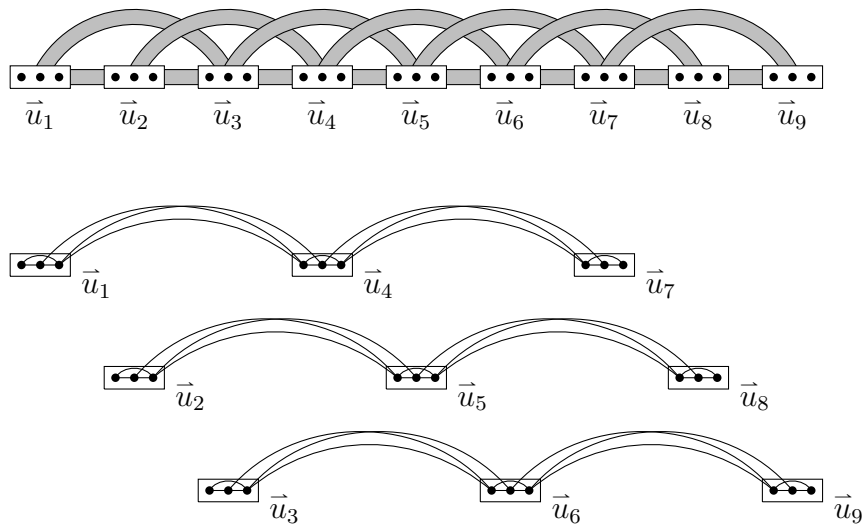

\begin{proof}[Proof of Proposition~\ref{m-path decomposition}]
Let $\seq{v} = (v_1, v_2, \ldots, v_{\l (k+1)t})$ be the vertices of $P=P^{k}_{(k+1)t}(\l)$. Consider the decomposition of $\seq{v}$ of the form $\seq{v} = (\seq{u}_1,\seq{u}_2,\ldots,\seq{u}_{(k+1)t})$, where each $\seq{u_i}$, $i=1,2,\ldots,(k+1)t$, is a segment of $\seq{v}$ of length $\ell$. With a small abuse of notation we will treat $\seq{u}_i$ either as a sequence, or as a set, depending on the context.

Now, for each $i = 1, 2 \ldots, k+1$, consider a subsequence $\seq{v}^{(i)} = (\seq{u}_i,\seq{u}_{i+(k+1)},\ldots,\seq{u}_{i+(t-1)(k+1)})$ of $\seq{v}$. Let $B_i$ be the copy of $B(\ell,r,t)$ on $\seq{v}^{(i)}$ in that ordering. In particular, each segment $\seq{u}_j$ induces a copy of $K_{\l}$ in $B_i$ and any two segments $\seq{u}_j$ and $\seq{u}_{j+(k+1)}$ in $B_i$ are $r$-bridged. Note also that the vertices of $\seq{v}^{(i)}$ form an independent set in $P$, hence the graph $B_i$ is edge-disjoint from $P$. Now put $B = B_1 \cup \ldots \cup B_{k+1}$.
Since for any $1 \leq i < i' \leq k+1$, $\seq{v}^{(i)}$ and $\seq{v}^{(i')}$ are disjoint, the graphs $B_i$ and $B_{i'}$ are vertex-disjoint and $B$ is a copy of $(k+1)B(\l,r,t)$, which is edge-disjoint from $P$.

In order to finish the proof of (\ref{P}) it is enough to show that any vertex in $\seq{v}$ is connected in $P \dcup B$ with $m$ consecutive vertices. To this end, take a vertex $v_{i\l + j}$, where $0 \le i \le(k+1)t - 1$ and $1 \leq j \leq \l$, and note that $v_{i\l + j} \in \seq{u}_{i+1}$. Then $v_{i\l + j}$ is connected in $B$ with $\l - j$ vertices $v_{i\l + j + 1},\ldots, v_{i\l +\l}\in \seq{u}_{i+1}$, as $\seq{u}_{i+1}$ induces in $B$ a clique $K_{\l}$. Moreover, since in $P$ the sets $\seq{u}_{i+1}, \seq{u}_{i+2},\ldots, \seq{u}_{i+k+1}$ induce a complete $(k+1)$-partite graph, $v_{i\l+j}\in\seq{u}_{i+1}$ is connected in $P$ with $k\l$ vertices from $\seq{u}_{i+2},\ldots, \seq{u}_{i+k+1}$, that is with $v_{(i+1)\l + 1}, v_{(i+1)\l + 2},\ldots, v_{(i+k)\l + \l}$. If $\l - j \ge r$, then the above two groups of vertices give us $\l - j + k\l \ge m$ consecutive neigbours. Otherwise, for the last $m - k\l - (\l - j)$ vertices, the connections are given by the edges of the $r$-bridge between $\seq{u}_{i+1}$ and $\seq{u}_{i+k+2}$ in $B$. For an illustration of (\ref{P}) see Fig.\,\ref{fig:completion}.




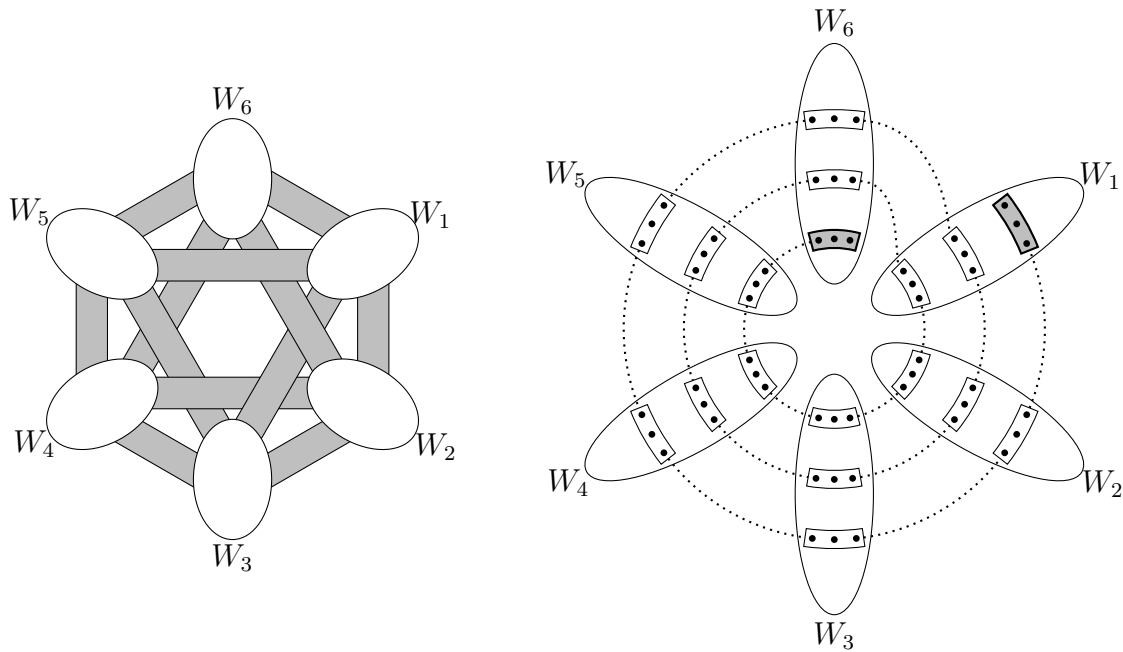
\begin{figure}
\centering
\begin{tikzpicture}[scale=0.4]

\coordinate(v1) at (84:7cm);
\coordinate(v2) at (90:7cm);
\coordinate(v3) at (96:7cm);

\draw (82:6.7cm) -- (82:7.3cm)
arc (82:98:7.3cm) -- (98:6.7cm)
arc (98:82:6.7cm) -- cycle;

\coordinate(v4) at (144:7cm);
\coordinate(v5) at (150:7cm);
\coordinate(v6) at (156:7cm);

\draw[rotate=60] (82:6.7cm) -- (82:7.3cm)
arc (82:98:7.3cm) -- (98:6.7cm)
arc (98:82:6.7cm) -- cycle;

\coordinate(v7) at (204:7cm);
\coordinate(v8) at (210:7cm);
\coordinate(v9) at (216:7cm);

\draw[rotate=120] (82:6.7cm) -- (82:7.3cm)
arc (82:98:7.3cm) -- (98:6.7cm)
arc (98:82:6.7cm) -- cycle;

\coordinate(v10) at (264:7cm);
\coordinate(v11) at (270:7cm);
\coordinate(v12) at (276:7cm);

\draw[rotate=180] (82:6.7cm) -- (82:7.3cm)
arc (82:98:7.3cm) -- (98:6.7cm)
arc (98:82:6.7cm) -- cycle;

\coordinate(v13) at (324:7cm);
\coordinate(v14) at (330:7cm);
\coordinate(v15) at (336:7cm);

\draw[rotate=240] (82:6.7cm) -- (82:7.3cm)
arc (82:98:7.3cm) -- (98:6.7cm)
arc (98:82:6.7cm) -- cycle;

\coordinate(v16) at (24:7cm);
\coordinate(v17) at (30:7cm);
\coordinate(v18) at (36:7cm);

\draw[rotate=300,thick,fill=lightgray] (82:6.7cm) -- (82:7.3cm)
arc (82:98:7.3cm) -- (98:6.7cm)
arc (98:82:6.7cm) -- cycle;

\coordinate(v19) at (83:5cm);
\coordinate(v20) at (90:5cm);
\coordinate(v21) at (97:5cm);

\draw (80:4.7cm) -- (80:5.3cm)
arc (80:100:5.3cm) -- (100:4.7cm)
arc (100:80:4.7cm) -- cycle;

\coordinate(v22) at (143:5cm);
\coordinate(v23) at (150:5cm);
\coordinate(v24) at (157:5cm);

\draw[rotate=60] (80:4.7cm) -- (80:5.3cm)
arc (80:100:5.3cm) -- (100:4.7cm)
arc (100:80:4.7cm) -- cycle;

\coordinate(v25) at (203:5cm);
\coordinate(v26) at (210:5cm);
\coordinate(v27) at (217:5cm);

\draw[rotate=120] (80:4.7cm) -- (80:5.3cm)
arc (80:100:5.3cm) -- (100:4.7cm)
arc (100:80:4.7cm) -- cycle;

\coordinate(v28) at (263:5cm);
\coordinate(v29) at (270:5cm);
\coordinate(v30) at (277:5cm);

\draw[rotate=180] (80:4.7cm) -- (80:5.3cm)
arc (80:100:5.3cm) -- (100:4.7cm)
arc (100:80:4.7cm) -- cycle;

\coordinate(v31) at (323:5cm);
\coordinate(v32) at (330:5cm);
\coordinate(v33) at (337:5cm);

\draw[rotate=240] (80:4.7cm) -- (80:5.3cm)
arc (80:100:5.3cm) -- (100:4.7cm)
arc (100:80:4.7cm) -- cycle;

\coordinate(v34) at (23:5cm);
\coordinate(v35) at (30:5cm);
\coordinate(v36) at (37:5cm);

\draw[rotate=300] (80:4.7cm) -- (80:5.3cm)
arc (80:100:5.3cm) -- (100:4.7cm)
arc (100:80:4.7cm) -- cycle;

\coordinate(v37) at (80:3cm);
\coordinate(v38) at (90:3cm);
\coordinate(v39) at (100:3cm);

\draw[thick,fill=lightgray] (75:2.7cm) -- (75:3.3cm)
arc (75:105:3.3cm) -- (105:2.7cm)
arc (105:75:2.7cm) -- cycle;

\coordinate(v40) at (140:3cm);
\coordinate(v41) at (150:3cm);
\coordinate(v42) at (160:3cm);

\draw[rotate=60] (75:2.7cm) -- (75:3.3cm)
arc (75:105:3.3cm) -- (105:2.7cm)
arc (105:75:2.7cm) -- cycle;

\coordinate(v43) at (200:3cm);
\coordinate(v44) at (210:3cm);
\coordinate(v45) at (220:3cm);

\draw[rotate=120] (75:2.7cm) -- (75:3.3cm)
arc (75:105:3.3cm) -- (105:2.7cm)
arc (105:75:2.7cm) -- cycle;

\coordinate(v46) at (260:3cm);
\coordinate(v47) at (270:3cm);
\coordinate(v48) at (280:3cm);

\draw[rotate=180] (75:2.7cm) -- (75:3.3cm)
arc (75:105:3.3cm) -- (105:2.7cm)
arc (105:75:2.7cm) -- cycle;

\coordinate(v49) at (320:3cm);
\coordinate(v50) at (330:3cm);
\coordinate(v51) at (340:3cm);

\draw[rotate=240] (75:2.7cm) -- (75:3.3cm)
arc (75:105:3.3cm) -- (105:2.7cm)
arc (105:75:2.7cm) -- cycle;

\coordinate(v52) at (20:3cm);
\coordinate(v53) at (30:3cm);
\coordinate(v54) at (40:3cm);

\draw[rotate=300] (75:2.7cm) -- (75:3.3cm)
arc (75:105:3.3cm) -- (105:2.7cm)
arc (105:75:2.7cm) -- cycle;

\begin{pgfonlayer}{front}
		\foreach \i in {v1,v2,v3,v4,v5,v6,v7,v8,v9,v10,v11,v12,v13,v14,v15,v16,v17,v18,
					v19,v20,v21,v22,v23,v24,v25,v26,v27,v28,v29,v30,v31,v32,v33,v34,v35,v36,
					v37,v38,v39,v40,v41,v42,v43,v44,v45,v46,v47,v48,v49,v50,v51,v52,v53,v54} \fill (\i) circle (3.2pt);		\end{pgfonlayer}	
					
\draw (90:5.5cm) ellipse (1.3cm and 4cm);
\draw [rotate=60] (90:5.5cm) ellipse (1.3cm and 4cm);
\draw [rotate=120] (90:5.5cm) ellipse (1.3cm and 4cm);
\draw [rotate=180] (90:5.5cm) ellipse (1.3cm and 4cm);
\draw [rotate=240] (90:5.5cm) ellipse (1.3cm and 4cm);
\draw [rotate=300] (90:5.5cm) ellipse (1.3cm and 4cm);

\draw[thick,dotted] (98:7cm) arc (98:142:7cm);
\draw[rotate=60,thick,dotted] (98:7cm) arc (98:142:7cm);
\draw[rotate=120,thick,dotted] (98:7cm) arc (98:142:7cm);
\draw[rotate=180,thick,dotted] (98:7cm) arc (98:142:7cm);
\draw[rotate=240,thick,dotted] (98:7cm) arc (98:142:7cm);

\draw[thick,dotted] (100:5cm) arc (100:140:5cm);
\draw[rotate=60,thick,dotted] (100:5cm) arc (100:140:5cm);
\draw[rotate=120,thick,dotted] (100:5cm) arc (100:140:5cm);
\draw[rotate=180,thick,dotted] (100:5cm) arc (100:140:5cm);
\draw[rotate=240,thick,dotted] (100:5cm) arc (100:140:5cm);

\draw[thick,dotted] (105:3cm) arc (105:135:3cm);
\draw[rotate=60,thick,dotted] (105:3cm) arc (105:135:3cm);
\draw[rotate=120,thick,dotted] (105:3cm) arc (105:135:3cm);
\draw[rotate=180,thick,dotted] (105:3cm) arc (105:135:3cm);
\draw[rotate=240,thick,dotted] (105:3cm) arc (105:135:3cm);

\draw[thick,dotted] (82:7cm) edge[out=355,in=95] (40:5cm);
\draw[thick,dotted] (80:5cm) edge[out=355,in=95] (45:3cm);

\node (V_1) at (30:10.2cm) {$W_1$};
\node (V_2) at (330:10.3cm) {$W_2$};
\node (V_3) at (270:10.2cm) {$W_3$};
\node (V_4) at (210:10.2cm) {$W_4$};
\node (V_5) at (150:10.3cm) {$W_5$};
\node (V_6) at (90:10.2cm) {$W_6$};

\begin{scope}[shift={(-20cm,0)}]

\draw[fill=lightgray] (90:6cm) -- (150:6cm) -- (150:4.8cm) -- (90:4.8cm) -- (90:6cm);
\draw[rotate=60,fill=lightgray] (90:6cm) -- (150:6cm) -- (150:4.8cm) -- (90:4.8cm) -- (90:6cm);
\draw[rotate=120,fill=lightgray] (90:6cm) -- (150:6cm) -- (150:4.8cm) -- (90:4.8cm) -- (90:6cm);
\draw[rotate=180,fill=lightgray] (90:6cm) -- (150:6cm) -- (150:4.8cm) -- (90:4.8cm) -- (90:6cm);
\draw[rotate=240,fill=lightgray] (90:6cm) -- (150:6cm) -- (150:4.8cm) -- (90:4.8cm) -- (90:6cm);
\draw[rotate=300,fill=lightgray] (90:6cm) -- (150:6cm) -- (150:4.8cm) -- (90:4.8cm) -- (90:6cm);

\draw[fill=lightgray] (90:5.3cm) -- (210:5.3cm) -- (210:3.2cm) -- (90:3.2cm) -- (90:5.3cm);
\draw[rotate=60,fill=lightgray] (90:5.3cm) -- (210:5.3cm) -- (210:3.2cm) -- (90:3.2cm) -- (90:5.3cm);
\draw[rotate=120,fill=lightgray] (90:5.3cm) -- (210:5.3cm) -- (210:3.2cm) -- (90:3.2cm) -- (90:5.3cm);
\draw[rotate=180,fill=lightgray] (90:5.3cm) -- (210:5.3cm) -- (210:3.2cm) -- (90:3.2cm) -- (90:5.3cm);
\draw[rotate=240,fill=lightgray] (90:5.3cm) -- (210:5.3cm) -- (210:3.2cm) -- (90:3.2cm) -- (90:5.3cm);
\draw[rotate=300,fill=lightgray] (90:5.3cm) -- (210:5.3cm) -- (210:3.2cm) -- (90:3.2cm) -- (90:5.3cm);

\draw[fill=white] (90:5cm) ellipse (1.3cm and 2cm);
\draw[fill=white] [rotate=60] (90:5cm) ellipse (1.3cm and 2cm);
\draw[fill=white] [rotate=120] (90:5cm) ellipse (1.3cm and 2cm);
\draw[fill=white] [rotate=180] (90:5cm) ellipse (1.3cm and 2cm);
\draw[fill=white] [rotate=240] (90:5cm) ellipse (1.3cm and 2cm);
\draw[fill=white] [rotate=300] (90:5cm) ellipse (1.3cm and 2cm);

\node (W_1) at (30:7.6cm) {$W_1$};
\node (W_2) at (330:7.8cm) {$W_2$};
\node (W_3) at (270:7.6cm) {$W_3$};
\node (W_4) at (210:7.6cm) {$W_4$};
\node (W_5) at (150:7.8cm) {$W_5$};
\node (W_6) at (90:7.6cm) {$W_6$};

\end{scope}

\end{tikzpicture}

\caption{For $k=2, \l=3, r=1, t=6, m=7$, the blow-up $C_6^2(9)$ (on the left) contains a copy of a path $P_{18}^2(3)$ (on the right) which is rolled up around the cycle. }
\end{figure}\label{fig:completion-cycle}

To prove (\ref{C}), let $W_1, W_2, \ldots, W_{2k+2}$ be the partition classes of $C$. Split each $W_i$, $i=1,2,\ldots,2k+2$, into $t/2$ subsets of size $\l$ and order them arbitrarily into segments $\seq{u}_i, \seq{u}_{i+2k+2},\ldots, \seq{u}_{i+(t/2-1)(2k+2)}$. Put $\seq{v} = (\seq{u}_1, \seq{u}_2,\ldots,\seq{u}_{t(k+1)})$ and note that this gives us the ordering of the desired graph $P$. Indeed, each of the segments $\seq{u}_i$ is an independent set in $C$ of size $\l$. Moreover, for any two segments $\seq{u}_i$ and $\seq{u}_j$, with $1\le i < j \le t(k+1)$, $j-i \leq k$, $\seq{u}_i$ and $\seq{u}_j$ induce in $C$ a complete bipartite graph. As for the rest of the proof, one can repeat the construction of $B$ as in the case (\ref{P}). Note that for each $i = 1, 2 \ldots, k+1$, the subsequence $\seq{v}^{(i)} = (\seq{u}_i,\seq{u}_{i+(k+1)},\ldots,\seq{u}_{i+(t-1)(k+1)})$ forms an independent set in $C$, since it contains only the vertices of ``antipodal'' partition sets of $C$, i.e. of $W_i$ and $W_{i+k+1}$. Thus, $B$ and $C$ are edge-disjoint and the rest of the proof goes along the same line as the proof of (\ref{P}). For an illustration of inclusion (\ref{C}) see Fig.\,\ref{fig:completion-cycle} and Fig.\,\ref{fig:completion-cycle-braid}.
\end{proof}

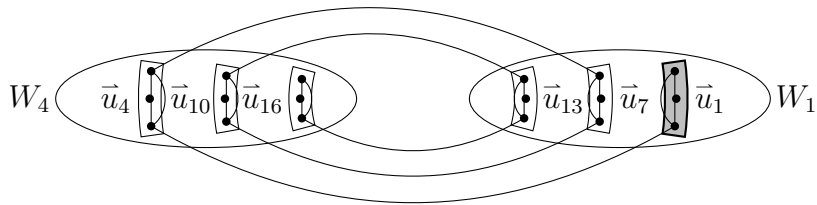
\begin{figure}
\centering
\begin{tikzpicture}[scale=0.5]

\coordinate(v7) at (174:7cm);
\coordinate(v8) at (180:7cm);
\coordinate(v9) at (186:7cm);
\draw[rotate=90] (82:6.7cm) -- (82:7.3cm)
arc (82:98:7.3cm) -- (98:6.7cm)
arc (98:82:6.7cm) -- cycle;

\coordinate(v16) at (354:7cm);
\coordinate(v17) at (0:7cm);
\coordinate(v18) at (6:7cm);

\draw[rotate=270,thick,fill=lightgray] (82:6.7cm) -- (82:7.3cm)
arc (82:98:7.3cm) -- (98:6.7cm)
arc (98:82:6.7cm) -- cycle;

\coordinate(v25) at (173:5cm);
\coordinate(v26) at (180:5cm);
\coordinate(v27) at (187:5cm);

\draw[rotate=90] (80:4.7cm) -- (80:5.3cm)
arc (80:100:5.3cm) -- (100:4.7cm)
arc (100:80:4.7cm) -- cycle;

\coordinate(v34) at (353:5cm);
\coordinate(v35) at (0:5cm);
\coordinate(v36) at (7:5cm);

\draw[rotate=270] (80:4.7cm) -- (80:5.3cm)
arc (80:100:5.3cm) -- (100:4.7cm)
arc (100:80:4.7cm) -- cycle;

\coordinate(v43) at (170:3cm);
\coordinate(v44) at (180:3cm);
\coordinate(v45) at (190:3cm);

\draw[rotate=90] (75:2.7cm) -- (75:3.3cm)
arc (75:105:3.3cm) -- (105:2.7cm)
arc (105:75:2.7cm) -- cycle;

\coordinate(v52) at (350:3cm);
\coordinate(v53) at (0:3cm);
\coordinate(v54) at (10:3cm);

\draw[rotate=270] (75:2.7cm) -- (75:3.3cm)
arc (75:105:3.3cm) -- (105:2.7cm)
arc (105:75:2.7cm) -- cycle;

\begin{pgfonlayer}{front}
		\foreach \i in {v7,v8,v9,v16,v17,v18,v25,v26,v27,v34,v35,v36,v43,v44,v45,v52,v53,v54} \fill (\i) circle (3.2pt);		\end{pgfonlayer}

\draw (0:5.5cm) ellipse (4cm and 1.3cm);
\draw [rotate=180] (0:5.5cm) ellipse (4cm and 1.3cm);

\node (V_1) at (0:10.2cm) {$W_1$};
\node (V_4) at (180:10.2cm) {$W_4$};

\draw (v7) edge[out=330,in=30] (v9);
\draw (v7) -- (v9);
\draw (v25) edge[out=330,in=30] (v27);
\draw (v25) -- (v27);
\draw (v43) edge[out=330,in=30] (v45);
\draw (v43) -- (v45);

\draw (v16) edge[out=150,in=210] (v18);
\draw (v16) -- (v18);
\draw (v34) edge[out=150,in=210] (v36);
\draw (v34) -- (v36);
\draw (v52) edge[out=150,in=210] (v54);
\draw (v52) -- (v54);

\draw (v16) edge[out=210,in=330] (v9);
\draw (v34) edge[out=210,in=330] (v27);
\draw (v52) edge[out=210,in=330] (v45);
\draw (v7) edge[out=30,in=150] (v36);
\draw (v25) edge[out=30,in=150] (v54);

\node (U_1) at (7.9cm,0) {$\seq{u}_1$};
\node (U_7) at (5.9cm,0) {$\seq{u}_7$};
\node (U_13) at (4cm,0) {$\seq{u}_{13}$};
\node (U_4) at (-7.9cm,0) {$\seq{u}_4$};
\node (U_10) at (-5.9cm,0) {$\seq{u}_{10}$};
\node (U_16) at (-4cm,0) {$\seq{u}_{16}$};

\end{tikzpicture}

\caption{Between any two ``antipodal'' partition classes of $C_6^2(9)$ there is a copy of the braid graph $B(3,1,6)$.}\label{fig:completion-cycle-braid}

\end{figure}

\subsection{An application of Janson's Inequality} Here we apply  Theorem \ref{Janson2} to  the graph $(k+1)B(\l,r,t)$ defined in the  previous subsection. Recall that functions $\Psi_G$ and $\Phi_G$ are defined in Section \ref{rg}.

\begin{prop}\label{Janson} Let $\tau>0$, $\l\ge r(r+1)\ge 2$ and $p\ge Cn^{-2/\l}$, where $C\ge 1$.
Further, let $B=(k+1)B(\l,r,t)$, $F$ be a subgraph of $B$ containing $K_\ell$, and $\cF$ be a family of at least $\tau n^{v_F}$ copies of $F$ in $K_n$.
Let $X$ be the number of copies of $F$ belonging to $\cF$ which are present in $G(n,p)$.
There exists a constant $c=c_{F}$ such that
\[
\PP(X \le\tau\Psi_F/2)\le \exp\{-\tau^2cCn\}.
\]
\end{prop}

\proof[Proof of Proposition \ref{Janson}]
We are going to show that $\Phi_F\ge\Phi_B\ge Cn$, where the first inequality is trivial. This, in view of Theorem \ref{Janson2}, implies Proposition \ref{Janson} with $c=(8 \cdot 4^{e_F})^{-1}$. We begin with a purely structural result.

Given a graph $G$ with $e_G\ge2$, let $d_G=\tfrac{e_G}{v_G-1}$ and set $d_{K_1}=0$. Then, define
\[
m_G=\max_{H\subseteq G}d_H.
\]
We claim that under the assumption $\l\ge r(r+1)$,
\begin{equation}\label{md}
m_B=d_{K_\l}=\frac\l2.
\end{equation}
To prove (\ref{md}), let $B'$ have the largest number of vertices among all subgraphs of $B$ which achieve the maximum in the definition of $m_B$. It is easy to check that $B'$ is connected and thus $B'\subseteq B(\l,r,t)$. Indeed, in general, if $G_1$ and $G_2$ are two vertex-disjoint graphs, then
\[
d_{G_1\cup G_2}=\frac{e_{G_1}+e_{G_2}}{v_{G_1}+v_{G_2}-1}<\frac{e_{G_1}+e_{G_2}}{v_{G_1}+v_{G_2}-2}\le\max_{1\le i\le2}\frac{e_{G_i}}{v_{G_i}-1}=\max_{1\le i\le2}d_{G_i}.
\]

Let $K^{(1)},\dots,K^{(t)} \subseteq $ be the $\ell$-cliques of $B(\l,r,t)$ as defined in Definition~\ref{bra}.
Let $B'$ intersect some $t'$ of them, respectively, in $s_{i_1},\dots,s_{i_{t'}}$ vertices.
Our goal is to show that $d_{B'}\le\l/2$, or equivalently, $e_{B'}-\tfrac \l2(v_{B'}-1)\le0.$ Note that $B'$ intersects at most $t'-1$ bridges, each in at most $\binom{r+1}2$ edges. This, together with inequalities $s_{i_j}\le \l$ and $r(r+1)\le\l$, implies that
\[
e_{B'}-\frac \l2(v_{B'}-1)\le\sum_{j=1}^{t'}\binom{s_{i_j}}2+(t'-1)\binom{r+1}2-\frac\l2\left(\sum_{j=1}^{t'}(s_{i_j}-1)+(t'-1)\right)\le0,
\]
which proves (\ref{md}).

Finally,
\[
\Phi_B=\min_{H\subseteq B,e_H>0}\Psi_H=n\min_H n^{v_H-1}p^{e_H}=n\min_H\left(np^{\frac{e_H}{v_H-1}}\right)^{v_H-1}\ge n\left(np^{m_B} \right)\ge C^{\l/2}n\ge Cn,
\]
where the first inequality uses the bounds $v_H\ge2$ and $np^{\frac{e_H}{v_H-1}}\ge np^{m_B}\ge C^{\l/2}\ge1$. \qed

\subsection{Subgraphs in dense graphs and hypergraphs}

In this subsection we quote several extremal results which guarantee the presence of copies of a given subgraph in a dense graph or hypergraph.
The first of them is the following supersaturation result of Erd\H os and Simonovits from~\cite{ES1983}. Recall that $\chi(F)$ denotes the \emph{chromatic number} of a graph $F$.

\begin{lemma}[\cite{ES1983}]\label{super} Let $k\ge3$ and $F$ be a graph with chromatic number $\chi(F) = k$. For every $\eps>0$ there exist $\beta>0$ and $n_0$ such that if a graph $G$ with $n\ge n_0$ vertices has at least
\[
\left(\frac{k-2}{k-1}+\eps\right)\binom n2
\]
edges, then $G$ contains at least $\beta n^{v_F}$ copies of $F$.
\end{lemma}

A related result we are going to use in Section 8 was proved by Alon and Yuster in~\cite{AY1996}. For graphs $G$ and $F$, we say that $G$ has \emph{an $F$-factor} if $G$ contains $\lfloor v_G/v_F\rfloor$ vertex-disjoint copies of $F$.
\begin{theorem}\label{AY}
For  every  $\eps>0$ and  for  every  graph  $F$  there exists  a  $T_0=T_0(\eps,F)$ such  that   for every  $T\ge T_0$, any   graph   $G$   with   $T$   vertices   and   with   minimum   degree $\delta(G)\ge(1-1/\chi(F)+\eps)T$  has an  $F$-factor.
\end{theorem}

As our proof of Covering Lemma~\ref{prop:covering} is based on the Regularity Method, we need the following two well-known results.
The first of them is a version of Szemer\'edi's Regularity Lemma~\cite{Sz1978} (see also Section 7.2 in~\cite{DRRS}). For two real numbers~$\delta>0$ and $d\in [0, 1]$,  given a graph~$G$ and two nonempty disjoint sets $A, B\subseteq V(G)$, we say that the pair $(A, B)$ is {\it $(\delta, d)$-quasirandom} if for all $X\subseteq A$ and $Y\subseteq B$ the inequality
\[
\big|e(X, Y)-d|X||Y|\big|\le \delta |A||B|
\]
holds, where $e(X,Y)$ is the number of edges with one endpoint in $X$ and the other in $Y$.
The pair $(A, B)$ is {\it $\delta$-quasirandom} if it is $(\delta, d)$-quasirandom for $d=e(A, B)/|A||B|$. This last quantity is called \emph{the density}  of the pair $(A,B)$ in $G$. In Section \ref{sec:cover} we will need the following simple observation the proof of which is left as an exercise.

\begin{fact}\label{heredit}
If $(A, B)$ is a $\delta$-quasirandom pair in $G$ and $A'\subset A$, $B'\subset B$, with $|A'|\ge\alpha|A|$, $|B'|\ge \beta|B|$ for some $\alpha,\beta>0$, then $(A',B')$ is $\delta'$-quasirandom  in $G$ with $\delta'=\tfrac{2\delta}{\alpha\beta}$. \qed
\end{fact}

\begin{lemma}[Szemeredi's Regularity Lemma, \cite{Sz1978}]\label{regularity_lemma}
Given $\delta > 0$ and $T_0\in\mathbb{N}$ there exists an integer $T_1=T_1(\delta, T_0)$ such that every graph $G = (V,E)$ on $n \geq T_0$ vertices admits a partition
\[
V = V_0 \dcup V_1 \dcup \ldots \dcup V_T
\]
of its vertex set such that
\begin{enumerate}[label=\rmlabel]
\item $T\in [T_0,T_1]$, $|V_0| \leq \delta |V|$, $|V_1| = \ldots = |V_T|$, and
\item for every $i\in [T]$ the set $\{ j \in [T] \setminus \{i\}: (V_i,V_j) \text{ is not } \delta\text{-quasirandom}\}$ has size at most $\delta T$.
\end{enumerate}
\end{lemma}

A partition guaranteed by the above lemma will be referred to as \emph{$\delta$-quasirandom}. Once a quasirandom partition is established, one can easily count copies of a given subgraph in it.

\begin{lemma}[Counting Lemma]\label{counting_lemma}
Let $F$ be a graph with vertex set $[f]$ and let $G$ be another graph with vertex partition $V(G) = V_1 \dcup \ldots \dcup V_f$ such that $(V_i,V_j)$ is a $\delta$-quasirandom pair whenever $ij \in F$. Then the number of ordered copies of $F$ in $G$, i.e. the number of $f$-tuples $(v_1,\ldots,v_f) \in V_1 \times \ldots \times V_f$ such that $v_iv_j\in G$ whenever $ij\in F$, equals
\[
\left( \prod_{ij\in F} d_{ij} \pm e_F\delta\right) \prod_{i=1}^f |V_i|,
\]
where $d_{ij} = e(V_i,V_j)/|V_i||V_j|$.
\end{lemma}

The last two results quoted in this section  deal with $h$-uniform hypergraphs (or $h$-graphs, for short) which are collections of $h$-element sets on a given vertex set (for $h=2$ these are just graphs). The first one comes from \cite{E1964} (see Corollary on page 188). An $h$-graph $\cH$ is \emph{$h$-partite} if its vertex set can be partitioned into sets $V_1\cup\ldots\cup V_h$ in such a way that for every edge $e$ we have $|e\cap V_i|=1$ for each $1\le i\le h$. Let $\cK_{h}^{(h)}(q)$ denote the $h$-partite complete $h$-graph. Note that the number of edges of $\cK_{h}^{(h)}(q)$ is $q^h$.


\begin{lemma}[\cite{E1964}]\label{erkpar} For all $h,q\ge2$ and all $w\ge w_0(h,q)$, if $\cH$ is an $h$-partite $h$-graph with each partition set of size~$w$, and with at least
\[
\frac{3^hh^h}{w^{1/q^{h-1}}}w^h
\]
edges, then $\cH$ contains a copy of $\cK_{h}^{(h)}(q)$ with $q$ vertices in each partition class.
\end{lemma}

In ~\cite[Lemma~8]{RRS2007} a counting extension of Lemma \ref{erkpar} has been deduced from the proofs in \cite{E1964}. Here we quote this result with respect to unordered copies.

\begin{lemma}[\cite{RRS2007}]
\label{lem:erdos} For all integers $h\geq 2$ and $q\ge h+1$ and every $d>0$ there exist $\tau>0$ and $n_0$ such that for every $h$-graph~$\cH$ on $n\geq n_0$ vertices with $e_{\cH}\ge dn^h$, there are at least $ \tau n^{hq}$ copies of $\cK_{h}^{(h)}(q)$ in $\cH$.
\end{lemma}

This lemma has a very useful consequence for graphs. Recall Definition \ref{blo} and observe that for all $u_i\in U_i$, $i=1,\dots,h$, the subgraph of $F(t_1,\dots,t_h)$ induced by $\{u_1,\dots,u_h\}$ is isomorphic to $F$. If  $t_1=\ldots =t_h=:q$, then we denote by $\cF(F(q))$ the family of all $q^h$ such subgraphs.

\begin{cor}\label{blow}
For every integer $q\geq 2$, real $d>0$, and a graph $F$ there exist $\tau>0$ and $n_0$ such that the following holds.   Let  $G$ be a graph on
$n\geq n_0$ vertices  and let   $\cF$ be a family of copies of $F$ contained in $G$ of size $|\cF|\ge dn^{v_F}$. Then $G$  contains at least $ \tau n^{q v_{F}}$ copies  $F'(q)$ of
 the $q$-blow-up $F(q)$ of $F$ such that $\cF(F'(q))\subseteq \cF$.
\end{cor}

\proof Let $V(F)=\{v_1,\dots,v_h\}$. Consider an auxiliary $h$-uniform hypergraph $\cH$ on the vertex set $V(G)$, where each edge corresponds to a copy $F'\in \cF$. Take a random partition $\Pi=V_1 \cup V_2 \cup \ldots \cup V_{h}$ of $V(G)$, where each vertex chooses its vertex class independently with probability $1/h$. Let $\cH_\Pi$ be the (random) $h$-partite subhypergraph of $\cH$ consisting of only those edges of $\cH$ which correspond to the copies of $F\in\cF$  with $v_i\in V_i$, $i=1,\dots,h$.
Observe that   $\E (|\cH_\Pi|) = \frac{1}{h^h} |\cH|,$ hence, there exists a partition $\Pi_0$ for which $|\cH_{\Pi_0}| \geq \frac{1}{h^h}|\cH|$.
Notice that $|\cH_{\Pi_0}| \ge d' n^h$, where $d' =  h^{-h}d$.
By Lemma~\ref{lem:erdos} applied to $\cH:=\cH_{\Pi_0}$, 
for some $\tau>0$ there are at least $\tau n^{qh}$ copies of $K^{(h)}_h(q)$ in~$\cH$. Note that each such copy corresponds to a copy $F'(q)$ of the $q$-blow-up $F(q)$ of $F$ in $G$. By the construction of $\cH$, we do have $\cF(F'(q))\subseteq \cF$.
\qed

\subsection{Interlacing sequences}

Here we prove a technical result which turns out to be crucial in establishing the existence of many connectable $m$-tuples when proving Lemmas~\ref{prop:absorbing} and~\ref{prop:covering}.

\begin{dfn}\label{interlaces}
For a graph $G$, we say that a sequence $(x_1,\ldots,x_{k+1}) \in V(G)^{k+1}$ \textit{interlaces} with a sequence $(y_1,\ldots,y_{k+1}) \in V(G)^{k+1}$, if
\[
\forall i =1,\ldots,k+1\,:\, y_i \in N_G(x_i,\ldots, x_{k+1}, y_1,\ldots,y_{i-1}).
\]
\end{dfn}

\begin{rem}\label{rema}\rm The above definition and Definition \ref{connectable_tuples} are related via the  notion of blow-up. Indeed, if each $x_i$, $i=1,\dots,k$, from Definition \ref{interlaces} is blown-up to a set $X_i'=\{x_i^{(1)},\dots,x_i^{(\ell)}\}$, while $x_{k+1}$ to a set $X'_{k+1}=\{x_{k+1}^{(1)},\dots,x_{k+1}^{(r)}\}$, then each sequence consisting of one element from each set $X_i'$ interlaces with $(y_1,\ldots,y_{k+1})$ and, consequently, the  sequences $\seq{x}=(x_1^{(1)},\dots,x_1^{(\ell)},\dots,x_k^{(1)},\dots,x_k^{(\ell)},x_{k+1}^{(1)},\dots,x_{k+1}^{(r)})$ and $(y_1,\ldots,y_{k+1})$ satisfy the condition in Definition \ref{connectable_tuples}. Hence, the subsequent technical result can be viewed as a tool  for creating $\xi$-connectable $m$-tuples.
\end{rem}

\begin{prop}\label{st}
 For every $k\geq 1$, $\eps > 0$, and $s$, there is $t$ and $\xi>0$ such that the following holds. For every $n$-vertex graph $G$ with $\delta(G) \geq \left(\frac{k}{k+1} + \eps\right)n$ and for every sequence of  disjoint sets $X_1,\dots, X_{k+1}$ in $V(G)$ of sizes $|X_i|=t$, $i=1,\dots,k+1$, there exist subsets $X_i'\subset X_i$ of sizes $|X_i'|=s$, $i=1,\dots,k+1$, and a set $Y\subset V(G)^{k+1}$ of size $|Y|=\xi n^{k+1}$ such that every $(x_1,\ldots,x_{k+1})\in X_1'\times\ldots\times X_{k+1}'$ interlaces with every $(y_1,\ldots,y_{k+1})\in Y$. Consequently, every sequence of vertices consisting of $\l$ elements of $X_1'$, followed by $\l$ elements of $X_2'$, ..., followed by $\l$ elements of $X_k'$, followed by $r$  elements of $X_{r+1}'$ is $\xi$-connectable in $G$.
\end{prop}
\proof

Let us choose constants $t,t^{(1)},\dots,t^{(k)}$ satisfying
$$t\gg t^{(1)}\gg\ldots\gg t^{(k)}\gg t^{(k+1)} := s.$$
We are going to prove by induction on $j=1,\dots,k+1$ the following statement:
\begin{align}\label{prop:statement}
& \exists \xi_j>0,\;Y
^{(j)}\subset V^j ,\; |Y^{(j)}|\ge\xi_jn^j,\; \notag \\
& \exists X_i^{(j)}\subset X_i,\; |X_i^{(j)}|\ge t^{(j)},\;i=1,\dots,k+1, \text{ such that}\notag \\
&\forall (y_1,\dots,y_j)\in Y^{(j)},\;(x_1,\dots,x_{k+1})\in X_1^{(j)}\times\ldots\times X^{(j)}_{k+1} \\
&\forall i=1,\dots,j:\; y_i\in N(x_i,\dots,x_{k+1},y_1,\dots,y_{i-1}) \notag.
\end{align}
Clearly, for $j=k+1$ this is the statement of Proposition \ref{st} with $X_i'=X_i^{(k+1)}$, $i=1,\dots,k+1$, $Y=Y^{(k+1)}$ and $\xi=\xi_{k+1}$.

We begin with $j=1$. Let $T_1 = X_1\times\ldots\times X_{k+1}$ and $t_1 = |T_1| = t^{k+1}$.
For any sequence $(x_1, \ldots, x_{k+1})\in T_1$, using \eqref{claim1:N} with $J=\{x_1,\ldots,x_{k+1}\}$, there are at least $\eps n$ vertices in $N_G(x_1,\dots, x_{k+1})$.
Consider an auxiliary bipartite graph~$B$ between sequences $\seq{x}=(x_1,\dots, x_{k+1}) \in T_1$ and vertices $y_1\in V$, where an edge is drawn if $y_1\in N_G(x_1,\dots, x_{k+1})$. It is easy to show by a double counting argument that at least $\tfrac12\eps n$ vertices $y_1$ satisfy $deg_{B}(y_1)\ge\tfrac12\eps t_1$. Indeed, otherwise, we would have
\[
\eps t_1n\le |B|<\left(\tfrac12\eps n\right)\times t_1+ n\times \left(\tfrac12 \eps t_1\right)=\eps nt_1,
\]
a contradiction. Denote the set of such vertices by $Y_1$.

By the Pigeonhole Principle, there exists a subset $Y^{(1)}\subset Y_1$, $|Y^{(1)}|\ge \xi_1 n$, where $\xi_1=\tfrac12\eps /\binom{t_1}{\eps t_1/2}$, and a family $\cX_1\subseteq T_1$ of vectors $\seq{x}$, $|\cX_1|=\tfrac12\eps t_1$, such that for all $y_1\in Y^{(1)}$ and all $\seq{x}\in\cX_1$, we have $y_1\in N_G(\seq{x})$.

The family $\cX_1$ can be viewed as a $(k+1)$-partite $(k+1)$-uniform hypergraph.
Now we are going to apply Lemma~\ref{erkpar} to $\cH := \cX_1$ with $h:=k+1$, $q:=t^{(1)}$, and $w:=t$. To this end we choose
\[
t \ge \max\left\{ w_0(k+1,t^{(1)}), \left(\frac{2(3(k+1))^{k+1}}{\eps}\right)^{(t^{(1)})^k} \right\},
\]
where the second parameter guarantees that $\cX_1$ is large enough so as to satisfy the assumptions of Lemma~\ref{erkpar}.
Hence, $\cX_1$ contains a $(k+1)$-uniform clique $K^{(k+1)}_{k+1}(t^{(1)})$.
Let $X_i^{(1)}$, $i=1,\dots,k+1$, be the vertex classes of that clique.
This completes the proof of the base step $j=1$.

Now assume that~\eqref{prop:statement} is true for some $j$, $1\le j\le k$. We will deduce that it is also true for $j+1$.
For each sequence
$$\seq{y}=(y_1,\dots,y_j)\in Y^{(j)},$$
consider an auxiliary bipartite graph $B:=B(\seq{y})$ between sequences $(x_{j+1},\dots, x_{k+1})\in T_{j+1}$, where $T_{j+1} =  X_{j+1}^{(j)}\times\ldots\times X_{k+1}^{(j)}$, and vertices $y_{j+1}\in  V$, where an edge is drawn if $y_{j+1}\in N_G(x_{j+1},\dots, x_{k+1},y_1,\dots,y_j)$. Set $t_{j+1}=|T_{j+1}|=(t^{(j)})^{k+1-j}$.

Since, again by \eqref{claim1:N}, $|N_G(x_{j+1},\dots, x_{k+1},y_1,\dots,y_j)|\ge\eps n$, for all $x_i\in X_i^{(j)}$, $i=j+1,\dots, k+1$, the degree of $(x_{j+1},\dots, x_{k+1})$ in $B$ is at least $\eps n$. Thus, by a similar double counting argument as in case $j=1$,  there are at least $\frac{1}{2}\eps n$ vertices $y_{j+1}\in V$ with $deg_{B}(y_{j+1})\ge \tfrac12\eps t_{j+1}$. Denote the set of such vertices by $Y_{j+1}$.
Consequently, by the Pigeonhole Principle, there is a subset $Y_{j+1}'\subset Y_{j+1}$, $|Y'_{j+1}| = \xi'_{j+1} n$, for some $\xi_{j+1}'>0$, and a family  $\cX_{j+1}\subseteq T_{j+1}$ of vectors $\seq{x}=(x_{j+1},\dots,x_{k+1})$, $|\cX_{j+1}|=\tfrac12\eps t_{j+1}$, such that for all $y_{j+1}\in Y'_{j+1}$ and all $\seq{x}\in\cX_{j+1}$, we have $y_{j+1}\in N_G(x_{j+1},\dots, x_{k+1},y_1,\dots,y_j)$.

We apply Lemma~\ref{erkpar} to $\cX_{j+1}$ with $h:=k+1-j$, $q:=t^{(j+1)}$ and $w:=t^{(j)}$ obtaining, for $t^{(j)}$ sufficiently large with respect to $t^{(j+1)}$, that $\cX_{j+1}$ contains a clique $K(\seq{y}):=K^{(k+1-j)}_{k+1-j}(t^{(j+1)})$. (Note that for $j=k$ Lemma~\ref{erkpar} degenerates to singletons and we just take $K(\seq{y})=\cX_{k+1}$.)
Recall that $K(\seq{y})$ and $Y'_{j+1}:=Y'_{j+1}(\seq{y})$ depend on $\seq{y}=(y_1,\dots,y_j)$. Owing to the finiteness of $\cX_{j+1}$, we can still select a subset $\tilde Y^{(j)}\subset Y^{(j)}$ with $|\tilde Y^{(j)}|\ge \tilde\xi_jn^j$
and a clique $K\subseteq \cX_{j+1}$
such that for all $\seq{y}\in\tilde Y^{(j)}$, we have $K(\seq{y})=K$. Let $X_{j+1}^{(j+1)},\dots, X_{k+1}^{(j+1)}$ be the partition classes of~$K$. Additionally, let $X_1^{(j+1)} = X_1^{(j)},\dots, X_j^{(j+1)} = X_j^{(j)}$.
The sequence $X_1^{(j+1)},\dots, X_{k+1}^{(j+1)}$ together with the set
\[
Y^{(j+1)}=\{(y_1,\dots,y_{j+1}):\; (y_1,\dots,y_j)\in \tilde Y^{(j)},\; y_{j+1}\in Y'_{j+1}(y_1,\dots,y_j)\}
\]
and constant $\xi_{j+1}=\tilde\xi_j\times \xi'_{j+1}$,
satisfy~\eqref{prop:statement} for $j+1$.
Note that $|Y^{(j+1)}| \ge \xi_{j+1} n^{j+1}$.
This completes the inductive proof of~\eqref{prop:statement} and, thus, of Proposition \ref{st}.
\qed

\section{Connecting and Reservoir}\label{sec:connect}

Here we prove Lemma \ref{prop:reservoir}, but first we formulate  the Connecting Lemma which will be used inside the proof of Absorbing Lemma in the next section. Both lemmas proved in this section utilize yet another connecting lemma, Lemma~\ref{lem:dcon} below, proved as Lemma~4.1 in \cite{DRRS}, where, for convenience, \emph{$k$-walks} instead of $k$-paths are considered. Formally, by a \emph{$k$-walk} in a graph $G$ we mean a sequence of not necessarily distinct vertices but such that any $k+1$ consecutive vertices are distinct and form a clique in~$G$.

\begin{lemma}[\cite{DRRS}]\label{lem:dcon}
	For every integer $k\ge1$ and $\eps>0$ there exists some $\rho>0$ such that
	every $n$-vertex graph $G$ with $\delta(G)\geq (\tfrac{k}{k+1}+\eps)n$
	satisfies the following property.
	
	For all pairs of disjoint $k$-tuples $\seq{x}, \seq{x}'$ which induce cliques in $G$,
	the number of $k$-walks connecting~$\seq{x}$ and~$\seq{x}'$ with $\ell_k$ internal vertices
	is at least $\rho n^{\ell_k}$, where $\ell_k=(k+1)(2^{k+1}-2)$.
\end{lemma}

The Connecting Lemma is, in a sense, a simpler version of Lemma~\ref{prop:reservoir}, where no reservoir set $R$ is put aside.

\begin{lemma}[Connecting Lemma]\label{CL}
For  every $\eps > 0$ there exists $\xi>0$ such that for sufficiently large $C=C(\eps,\xi)$, every $n$-vertex graph $G$ with $\delta(G) \geq \left(\frac{k}{k+1} + \eps\right)n$, and $p = p(n)\ge Cn^{-2/\l}$, a.a.s.\ $H = G \cup G(n,p)$ has the following property.

Let $m = k\l+r$, with $\l\ge r(r+1)\ge2$. For every subset $Z\subseteq V$ with $|Z|\le \xi n /(2(k+1))$ and every pair of disjoint $\xi$-connectable $m$-tuples $\seq{x},\seq{x}' $ which induce cliques in $G$, there exists an $m$-path connecting $\seq{x} $ and $\seq{x}' $ with $\l(k+1)2^{k+1}$ internal vertices, all from $V\setminus Z$.
\end{lemma}

\begin{proof}
Let $\rho$ and $\ell_k$ be as in Lemma~\ref{lem:dcon}.
Choose $\xi \le \rho/(2(2^{k+1}-2))$.
Let $\seq{x} = (x_1,\ldots,x_m)$, $\seq{x}' = (x_1',\ldots,x_m')$ be $\xi$-connectable $m$-tuples. Fix $Z\subseteq V$ with $|Z|\le \xi n /(2(k+1))$ and put $L = \ell_k+2(k+1) = (k+1)2^{k+1}$. We will first show that there are $\Omega(n^L)$ $k$-walks  in $G$  with $L$ internal vertices, all avoiding $Z$, that connect $\seq{x}$ to $\seq{x}'$. (Formally, we connect the last $k$ vertices of $\seq{x}$ with the last $k$ vertices of $\seq{x}'$, so, with some abuse of terminology,  internal vertices are precisely those which are disjoint from the set $\{x_1,\ldots,x_m,x_1',\ldots,x_m'\}$.)

Indeed, consider ordered $(k+1)$-cliques $\seq{y} = (y_1,\ldots,y_{k+1}$), $\seq{y}'=(y_1',\ldots,y_{k+1}')$, as in Definition~\ref{connectable_tuples}, corresponding, respectively, to $\seq{x}$, $\seq{x}'$ which are disjoint from $Z$.
There are at least
\[
\left(\xi n^{k+1} - (k+1) |Z| n^{k} \right)^2 \ge \frac{1}{4}\xi^2 n^{2k+2},
\]
of them, since $|Z| \le \xi n /(2(k+1))$.

By Lemma~\ref{lem:dcon}, applied to the $k$-tuples $\seq{y}_{-}=(y_2,\dots,y_{k+1})$ and $\seq{y}'_{-}=(y_{2}',\dots, y_{k+1}')$, there exist
\[
\rho n^{\ell_k} - \ell_k (|Z| + 2m) n^{\ell_k-1} \ge  \frac{1}{2} \rho n^{\ell_k},
\]
$k$-walks connecting $\seq{y}_{-}$ and $\seq{y}'_{-}$, with $\ell_k$ internal vertices, all omitting~$Z$, $\seq{x}$, and $\seq{x}'$. Thus, altogether we have $\tfrac18\xi^2\rho n^L$ $k$-walks connecting $\seq{x}$ to $\seq{x}'$,  with $L$ internal vertices, all omitting~$Z$. Consequently, at least
\[
\frac18\xi^2\rho n^L-O(n^{L-1})\ge\frac1{10}\xi^2\rho n^L
\]
of them are $k$-paths. Let $\cP$ be the family of all such $k$-paths and $\cP_{int}$ the family of the sub-$k$-paths of the $k$-paths in $\cP$ spanned by  the $L$ internal vertices.

By Corollary \ref{blow} with $d=\frac1{10}\xi^2\rho$, $F=P^k_L$, $G=G[V\setminus Z]$, and $\cF=\cP_{int}$, for some $\tau'=\tau'(d)>0$, there are at least $\tau' n^{\l L}$ copies $P'(\ell)$ of the $\l$-blow-up $P^k_L(\ell)$ with $\cF(P'(\ell))\subseteq\cP_{int}$. We select from them at least $\tau n^{\l L}$ copies which have mutually distinct vertex sets, where $\tau=\tfrac{\tau'}{(\ell L)!}$. 
Let us consider a sequence of vertices $\seq{v}$ that begins with $\seq{x}$, ends with the reverse of $\seq{x}'$, and in between consists of the $\l L$ vertices of $P'(\ell)$ (the order in each $\l$-independent set obtained by the blow-up is fixed arbitrarily).

Notice that due to the choice of $\seq{y}$ and $\seq{y}'$, and the inclusion $\cF(P'(\ell))\subseteq\cP_{int}$, each vertex of $\seq{x}$ is already connected to the $m$ subsequent vertices of $\seq{v}$ and the same is true for $\seq{x}'$.
 Indeed, split the vector $\seq{x}$ into $k$ blocks of length $\l$ and  one block of length $r$. Then, each $x_i$ in the $j$th block, $j=1,\dots,k+1$, is adjacent to $m-i\ge m-j\ell$ elements lying in front of it in $\seq{x}$ plus $j\ell$ elements in the $\l$-blow-ups of the first $j$ elements of $\seq{y}$ (see Def.~\ref{connectable_tuples}). Thus, although the sequence $\seq{v}$ does not yet induce a full $m$-path, the only missing edges have all their vertices in $P'(\ell)$.

By Proposition~\ref{m-path decomposition}, we need to complement $P'(\ell)$ with a copy of $B = (k+1)B(\l,r,2^{k+1})$ in $G(n,p)$. For each $P'(\ell)$ let $B_{P'(\ell)}$ be the copy of $B$ which complements $P'(\ell)$ to a graph containing an $m$-path and let $\cB$ be the family of all such $B_{P'(\ell)}$.
 By Remark \ref{1-1}, we have $|\cB| \ge\tau n^{\l L}$. By Proposition~\ref{Janson}, there  exists $c=c_B>0$ such that with probability at least $1-\exp\{-\tau^2cCn\}$, at least one of them is present in $G(n,p)$, which yields the existence of an $m$-path connecting $\seq{x}$ and $\seq{x}'$ in $G\cup G(n,p)$ which avoids $Z$.
 As there are at most $n^{m}$ possibilities for the choice of each of~$\seq{x}$ and $ \seq{x}'$ and at most $2^n$ for $Z$, applying the union bound and taking $C=C(\tau,c)$ large enough, we conclude that a.a.s. the same it true for all choices of $Z$, $\seq{x}$, and $\seq{x}'$.
\end{proof}

For the proof of Lemma~\ref{prop:reservoir}, we need a modification of the notion of connectability.
\begin{dfn}\label{R-connectable_tuples}
Given $k\ge1$ and $\xi>0$, and a set $R$, an $m$-tuple $(x_1, x_2,\ldots,x_m)$ of vertices of a graph $G-R$ is \textit{$(R,\xi)$-connectable} if there exist $\xi |R|^{k+1}$  (ordered) copies $(y_1, y_2,\ldots,y_{k+1})$ of $K_{k+1}$ in $G[R]$ with the property that for each $i = 1, 2,\ldots,k+1$, $y_i\in N_G(x_{(i-1)l+1},\ldots,x_m)$.
\end{dfn}

\begin{proof}[Proof of Lemma~\ref{prop:reservoir}] Fix $\eps>0$ and $0<\xi<1$ and let $\rho:=\rho(\eps/2)$ be given by Lemma~\ref{lem:dcon}. Choose
\begin{equation}\label{gaxiro}
\gamma=\min\{\xi^2/{2^{2k+6}},\rho^2/4\}.
\end{equation}
Consider a subset $R\subseteq V$ chosen at
random by including each element of $V$ to $R$, independently, with probability $\gamma^2$. It is easy to see that a.a.s.\ $R$ satisfies the following three properties:
\begin{enumerate}[label=\rmlabel]
\item $\tfrac12\gamma^2n\le|R|\le2\gamma^2n$,
\item\label{enum:Rmindeg} $|N_G(v)\cap R|\ge \left(\frac k{k+1}+\frac\eps2\right)|R|$ for every $v\in V$, and
\item every $\xi$-connectable $m$-tuple in $G$ becomes $(R,\xi/2^{k+2})$-connectable.
\end{enumerate}
Indeed, $X=|R|$ is binomially distributed with $\E X=\gamma^2 n$, so the first property follows from Chebyshev's inequality.
Since $X_v=|N_G(v)\cap R|$ is also binomial with expectation $\gamma^2|N_G(v)|\ge \gamma^2(\tfrac k{k+1}+\eps)n$, the second property holds, simultaneously for all $v$, from Chernoff's
bound (see, e.g., \cite[Theorem 2.1]{JLR}).

To prove (iii), we employ a standard application of Janson's inequality (see, e.g., \cite[Theorem 2.14]{JLR}).
Given a  $\xi$-connectable $m$-tuple $\seq{x}$ in $G$, let $\cK$ be the family of $\xi n^{k+1}$ ordered copies of $K_{k+1}$ which witness the $\xi$-connectability  of $\seq{x}$. Let $Y:=Y_{\seq{x}}$ be the number of $K\in\cK$ which are contained in $R$. We apply the inequality in \cite[Theorem 2.14]{JLR} to $Y$ with $t:=\tfrac13\E Y$. Observe that $\E Y=\xi n^{k+1}\gamma^{2(k+1)}$, while $\overline{\Delta}=O(n^{2k+1})$.
Hence, $\PP(Y\le\tfrac23\E Y)\le\exp\{-\Omega(n)\}.$ This is so small that a.a.s.\ for all choices of $\seq{x}$ we have
\[
Y_{\seq{x}}\ge\frac23\xi\gamma^{2(k+1)}n^{k+1}\ge \frac1{2^{k+2}}\xi|R|^{k+1},
\]
where we also used the R-H-S of (i).
	
For the rest of the proof of Lemma~\ref{prop:reservoir} we fix one set $R\subseteq V$ having the above three properties. Let us now fix two ordered $\xi$-connectable $m$-tuples $\seq{x}, \seq{x}'$ in~$G-R$ as well as a subset $S\subseteq R$ with $|S|\le \sqrt\gamma|R|$.
	%
	%
We are going to show that with  probability very close to one, there is an $m$-path in $H$ connecting~$\seq{x}$ and $\seq{x}'$ with~$\l(k+1)2^{k+1}$ internal vertices, all from $R\setminus S$.

To this end, note that due to  property (iii) of $R$, sequences $\seq{x}$ and $\seq{x}'$ are  $(R,\xi/2^{k+2})$-connectable.
Hence, one can extend $\seq{x}$ to an $m$-path~$\seq{x}\,\seq{y}$, where $\seq{y}$ is a $(k+1)$-tuple in $G[R\sm S]$ which `witnesses' the $(R,\xi)$-connectability of $\seq{x}$, in at least
\[
\frac1{2^{k+2}}\xi|R|^{k+1}-\sqrt\gamma|R|^{k+1}\overset{(\ref{gaxiro})}{\ge}\frac1{2^{k+3}}\xi|R|^{k+1}
\]
ways.
We  extend $\seq{x}'$ to $\seq{x}'\seq{y}'$ in a similar way.

In turn, by~\ref{enum:Rmindeg}, we are in position to apply Lemma~\ref{lem:dcon}. Recalling that $\ell_k = (k+1)(2^{k+1}-2)$, 
we obtain  at least
\[
\rho |R|^{\ell_k}-\sqrt\gamma|R|^{\l_k}\overset{(\ref{gaxiro})}{\ge}\frac12\rho|R|^{\ell_k}
\]
$k$-walks connecting $\seq{y}_{-}$ and $\seq{y}'_{-}$, with $\ell_k$ internal vertices, all belonging to $R\setminus S$.   Thus, altogether we have at least $2^{-(2k+7)}\xi^2\rho |R|^L$ $k$-walks connecting $\seq{x}$ to $\seq{x}'$,  with $L=(k+1)2^{k+1}$ internal vertices, all belonging to $R\setminus S$. Consequently, at least $2^{-(2k+8)}\xi^2\rho |R|^L$ of them are $k$-paths.

Let $\cP$ be the family of all such $k$-paths and $\cP_{int}$ -- the family of the sub-$k$-paths of the $k$-paths in $\cP$ spanned by  the $L$ internal vertices.
Similarly as in the proof of Lemma \ref{CL}, by Corollary \ref{blow} with $d=2^{-(2k+8)}\xi^2\rho$, $F=P^k_L$, $G=G[R\setminus S]$, and $\cF=\cP_{int}$, for some $\tau=\tau(d)>0$, there are at least $\tau |R|^{\l L}$
copies $P'(\ell)$ of the $\l$-blow-up $P^k_L(\ell)$ with $\cF(P'(\ell))\subseteq\cP_{int}$ and mutually distinct vertex sets.
As in the previous proof, each such copy misses a copy of $B=(k+1)B(\ell,r,2^{k+1})$ to close an $m$-path between $\seq{x}$ and $ \seq{x}'$.

By Proposition~\ref{Janson}, using also the L-H-S of (i), there exists $c=c_B>0$ such that with probability at least
\[
1-\exp\{-\tau^2cC|R|\}\ge1-\exp\{-\tau^2cC\gamma^2n/2\},
\]
at least one of them is present in $G(n,p)$. This yields the existence  in $G\cup G(n,p)$ of an $m$-path connecting $\seq{x}$ and $\seq{x}'$,  with~$\l L$ internal vertices, all from $R\setminus S$.
As there are at most $n^{m}$ possibilities for the choice of each of~$\seq{x}$ and $ \seq{x}'$ and at most $2^n$ for $S$, applying the union bound and taking $C:=C(\tau,\gamma,c)$ large enough, we conclude that a.a.s.\ the same is true for all choices of $S\subseteq R$, $\seq{x}$, and $\seq{x}'$.

\end{proof}

\section{Absorbing Path}\label{sec:absorb}

We build the absorbing path $A$ from small  blocks, called absorbers.

\begin{dfn}
Given $\xi>0$,  a graph $G$, and a vertex $v\in V:=V(G)$, a $2m$-tuple $(x_m,x_{m-1},\ldots,x_1,x_1',\ldots,x_{m-1}',x_m')\in V^{2m} $ is a \textit{half-$(\xi,v)$-absorber} in $G$ if
\begin{enumerate}[label=\rmlabel]
\item $x_1,x_2,\ldots,x_m,x_1',x_2',\ldots,x_m' \in N_G(v)$;
\item $\seq{x}=(x_1,x_2,\ldots,x_m)$, $\seq{x}'=(x_1',x_2',\ldots,x_m')$ are $\xi$-connectable in $G$;
\item\label{halfabsorber} $(x_m,x_{m-1},\ldots,x_1,x_1',\ldots,x_{m-1}',x_m')$ induces  in~$G$ an $(r,\l,\ldots,\l,r)$-blow-up of a~$P_{2k+2}^k$.
\end{enumerate}
If condition \ref{halfabsorber} is replaced by
\begin{enumerate}
\item[\textit{(iii)$'$}]\label{fullabsorber} $(x_m,x_{m-1},\ldots,x_1,x_1',\ldots,x_{m-1}',x_m')$ induces  in $H=G\cup G(n,p)$ an $m$-path,
\end{enumerate}
then we call the $2m$-tuple $(x_m,x_{m-1},\ldots,x_1,x_1',\ldots,x_{m-1}',x_m')$ a \textit{(full) $(\xi,v)$-absorber}.
A~$2m$-tuple which is a $(\xi,v)$-absorber for some $v\in V$ is called a \textit{$\xi$-absorber}.
\end{dfn}
 The key observation is that if $(x_m,x_{m-1},\ldots,x_1,x_1',\ldots,x_{m-1}',x_m')$ is a $(\xi,v)$-absorber, then $(x_m,x_{m-1},\ldots,x_1,v,x_1',\ldots,x_{m-1}',x_m')$ is an $m$-path (here we just need properties (i) and (iii)$'$,  not (ii)). This allows
 for including (or absorbing)  $v$ into a path or cycle which contains a $(\xi,v)$-absorber as a segment. To absorb an entire subset $U$ of vertices, we will need many disjoint $(\xi,u)$-absorbers for each $u\in U$. In fact, by a simply greedy argument, at least $|U|$ disjoint $(\xi,u)$-absorbers per each vertex $u$ would suffice.

The next result asserts that for some $\xi>0$ there are many half-$(\xi,v)$-absorbers for every $v\in V(G)$.

\begin{prop}\label{prop:half-v-absorbers}
For every $\eps>0$ there exist $\xi>0$ and $\beta'>0$ such that if $G$ is an $n$-vertex graph with $\delta(G)\geq (\tfrac{k}{k+1}+\eps)n$,
then, for every $v\in V(G)$, there are at least $\beta' n^{2m}$ half-$(\xi,v)$-absorbers.
\end{prop}

\begin{proof} Fix $\eps>0$. Let $\beta$ be given by Lemma~\ref{super}.
We are also going to apply Proposition~\ref{st} with $s:=\l$; let $t$ and $\xi$ be the resulting constants. Finally, let
\[
\beta'=\beta(\tfrac{k}{k+1}+\eps)^{(2k+2)t}.
\]

By \eqref{claim1:d}, for every $v$,
\[
\delta(G[N(v)])\ge\left(\frac{k-1}k+\eps\right)|N(v)|
\]
which implies that
\[
e(G[N(v)])\ge\left(\frac{k-1}k+\eps\right)\binom{|N(v)|}2.
\]
Since $\chi(P_{2k+2}^k)=k+1$, we also have $\chi(P_{2k+2}^k(t))=k+1$, where $P_{2k+2}^k(t)$ is the $t$-blow-up of the $k$-path $P_{2k+2}^k$ on $2k+2$ vertices.
Thus, by Lemma~\ref{super},  $G[N(v)]$ contains at least
\[
\beta |N(v)|^{(2k+2)t}\ge\beta' n^{(2k+2)t}
\]
copies of $P_{2k+2}^k(t)$.  Fix one such copy and let $X_{k+1}, \dots, X_1, \bar X_1, \dots, \bar X_{k+1}$ be its vertex classes. By two applications of Proposition~\ref{st} (with $s=\ell$), one to $X_{k+1}, \dots, X_{1}$, the other to $\bar X_1, \dots, \bar X_{k+1}$, we obtain subsets $X_{k+1}',\dots, X_{1}', \bar X_1', \dots, \bar X_{k+1}'\subseteq V$  and two sets of $(k+1)$-tuples $Y,\bar Y\subseteq V^{k+1}$ such that
 \begin{enumerate}
 \item $|X_i'| = |\bar X_i'| =\ell$ for $i=1,\dots,k$ while $|X_{k+1}'| = |\bar X_{k+1}'| = r$ (we delete arbitrary $\l-r$ vertices from the $(k+1)$-st subset guaranteed by Proposition~\ref{st});
 \item $|Y|,|\bar Y|\ge\xi n^{k+1}$;
 \item every $(x_1,\ldots,x_{k+1})\in X_1'\times\ldots\times X_{k+1}'$ interlaces with every $(y_1,\ldots,y_{k+1})\in Y$ as well as every $(\bar x_1,\ldots,\bar x_{k+1})\in \bar X_1'\times\ldots\times\bar X_{k+1}'$ interlaces with every $(\bar y_1,\ldots,\bar y_{k+1})\in \bar Y$.
 \end{enumerate}

To finish the proof, consider first an $m$-tuple $\seq{x}$ consisting of all the vertices of $X_1',\ldots,X_{k+1}'$, in this order.
By Proposition \ref{st} (see also Remark \ref{rema}),
 $\seq{x}$ is $\xi$-connectable.
By the same token, the  sequence  $\seq{x}'$ listing all the elements in sets $\bar X_1',\ldots,\bar X_{k+1}'$ is a $\xi$-connectable $m$-tuple. Hence, $(\seq{x})^{-1}\seq{x}'$ is a half-$(\xi,v)$-absorber. In summary, each of the  $\beta' n^{(2k+2)t}$  $t$-blow-ups of $P_{2k+2}^k$ generates  a half-$(\xi,v)$-absorber. On the other hand, each of the half-$(\xi,v)$-absorbers can be generated by  at most $n^{(2k+2)t - 2m}$ such blow-ups. Thus, the assertion follows by taking the ratio of the two quantities.
\end{proof}

Next, we analyze what is needed in order to get an $m$-path as in \textit{(iii)'} starting from a blow-up appearing in \textit{(iii)}.
Let $B:=(k+1)B(\l,r,2)$ and let $B^-$ be the graph consisting of a copy of $(k-1)B({\l},r,2)$ and two vertex disjoint copies of the disjoint union of $K_{\l}$ and $K_r$ joined by an $r$-bridge, that is, $B^-$ is obtained from $B$ by removing $\l-r$ vertices from two cliques $K_{\l}$ belonging to distinct copies of $B(\l,r,2)$.
Given $v$ and a half-$(\xi,v)$-absorber $\seq{x}$, there is a  copy  of $B^-$ which, if included in $G(n,p)$, completes in $H$ a $(\xi,v)$-absorber on  $\seq{x}$.

Let $X$ be a random variable which counts the number of copies of $B^-$ in $G(n,p)$  and, for any vertex $v$, let $X_v$ be the number of those of them which turn a half-$(\xi,v)$-absorber into a full $(\xi,v)$-absorber.
Notice that the number of vertices of $B^-$ is $2m$ and the number of edges is $2k\binom{\l}{2} + 2\binom{r}{2} + (k+1)\binom{r+1}{2}$. Thus, putting
\[
\Psi :=\Psi_{B^-}= n^{2m} p^{2k\binom{\l}2 + 2\binom{r}2 + (k+1)\binom{r+1}2},
\]
 we have $\E X\le \Psi$ and $\E X_v \geq \beta'\Psi$ (cf. Proposition~\ref{prop:half-v-absorbers}). Finally, let $Y$ be the number of  intersecting pairs of copies of $B^-$ in $G(n,p)$.

\begin{prop}\label{prop:absorbers}
Let  $\beta'$ be as in Proposition \ref{prop:half-v-absorbers} and $p = p(n) \ge Cn^{-2/\l}$ for sufficiently large constant $C\ge1$. There exists a constant $D:=D(k,r,\l)$ such that the following properties hold a.a.s.
\begin{enumerate}[label=\rmlabel]
\item $X \leq 2\Psi$;
\item $Y\le D\Psi^2/n$;
\item for each $v\in V(G)$, $X_v \geq \frac12 \beta' \Psi$.
\end{enumerate}
\end{prop}

\begin{proof}
Part (i): Since $B^-$ is a subgraph of $B$ containing $K_{\l}$, by the proof of Proposition~\ref{Janson}, $\Phi_{B^-} = \Psi_{K_{\l}} = n^{\l}p^{\binom{\l}2}\ge Cn$.
By Chebyshev's inequality
\[
\PP(X \ge 2\Psi) \le \PP(X \ge 2\E X) \le \PP(|X-\E X| \ge \E X)
\le \frac{\Var X}{(\E X)^2}=O(1/\Phi_{B^-})=o(1)
\]
(see the proof of Theorem 3.4 in \cite{JLR} and Remark 3.7 therein).

Part (ii) is also a consequence of Chebyshev's inequality, but more technical as it applies to the numbers of copies of several non-isomorphic graphs (all possible unions $F$ of pairs of intersecting copies of $B^-$.) However, we can just quote inequality (3.22) from~\cite{JLR}, page~76, which states that for every such $F$ the number $X_F$ of copies of $F$ in $G(n,p)$ a.a.s. satisfies the inequality $X_F\le D_F\Psi^2/\Phi^v_{B^-}$ for some constant $D_F$, where $\Phi_{B^-}^v = \min\{\Phi_{B^-},n\}$ (see also: Notes on Notation in \cite[page 10]{JLR}). Since $\Phi^v_{B^-}=n$,  (ii) follows with $D=\sum_FD_F$.

Part (iii) follows by Proposition \ref{Janson} with $\tau=\beta'$, $t=2$, $F:=B^-$, and $\cF$ -- the family of all copies of $B^-$ which turn a half-$(\xi,v)$-absorber into a full $(\xi,v)$-absorber.
Then, there exists a constant $c=c_{B^-}>0$ such that
\[
\PP\left(X_v \le \frac12 \beta' \Psi\right)\le \exp\{-(\beta')^2cCn\}
\]
and, taking $C = C(\beta',c)$ large enough, (iii) follows by the union bound over all $v$.
\end{proof}

\medskip


Using the assumptions on $p$ and $\l$, it can be easily checked that $\Psi=\Omega(n^{k+1})$. Thus, Proposition~\ref{prop:absorbers} (iii) guarantees a.a.s. $\Omega(\Psi)=\Omega(n^{k+1})$  $\xi$-absorbers in $H$. We now thin down this family to a linear size in $n$ in a random fashion.

\begin{prop}\label{A'}
Let $\gamma\le(\beta'/24D)^{2/3}$. Then there exists a family $\cA'$ of $\xi$-absorbers with the following properties:
\begin{enumerate}[label=\rmlabel]
\item $|\cA'| \leq 6\gamma^{3/2}n$;
\item the number of pairs of intersecting elements in $\cA'$ is at most $ \frac18\beta' \gamma^{3/2} n$;
\item for every $v\in V(G)$, there are at least $\frac14 \beta'\gamma^{3/2}  n$ $(\xi,v)$-absorbers in $\cA'$.
\end{enumerate}
\end{prop}
\proof
Put
\[
q :=\gamma^{3/2}n/\Psi\]
and denote by $\cA_q$ a random subfamily of $\xi$-absorbers which is obtained by selecting each one independently with probability $q$.
By Proposition~\ref{prop:absorbers}(i),
\[
\E|\cA_q|\le 2\Psi q=2\gamma^{3/2}n.
\]
Hence, by Markov's inequality
\[
\PP(|\cA_q| > 6\gamma^{3/2}n) \leq \frac1{3}.
\]
Similarly, by Proposition~\ref{prop:absorbers}(ii), the expected number  of pairs of intersecting elements in $\cA_q$ is at most $D\Psi^2/n$ and  thus the probability that their number is greater than $\tfrac18 \beta' \gamma^{3/2} n$ can be bounded from above by
\[
\frac{(D\Psi^2/n) q^2}{\beta' \gamma^{3/2} n/8} = \frac{8D\gamma^{3/2}}{\beta'} \leq \frac13.
\]

Finally, for a fixed $v\in V$, note that the number of $(\xi,v)$-absorbers in $\cA_q$ is binomially distributed. Hence, by Proposition \ref{prop:absorbers}(iii), its expectation is $X_vq\ge \frac12 \beta' \Psi q = \frac12 \beta' \gamma^{3/2}n$.
Thus, by Chernoff's bound (see, e.g., \cite[Theorem 2.1]{JLR}) and the union bound over all $v$, the probability of the opposite event to the one stated in (iii) is at most
\[
n\exp\{-\beta'\gamma^{3/2}n/8\}=n\exp\{-\Omega(n)\}<1/3,
\]
for sufficiently large $n$.
In conclusion, the probability that  properties (i)-(iii) hold for $\cA_q$ is positive, and thus, there exists a family $\cA'$ of $\xi$-absorbers which satisfies all three of them.
\qed

\medskip

\begin{proof}[Proof of Lemma~\ref{prop:absorbing}]
Given $\eps>0$, let $\beta'=\beta'(\eps)$ be as in Proposition~\ref{prop:half-v-absorbers}, and let $\xi=\min\{\xi_1,\xi_2\},$ where $\xi_1$ is as in Lemma~\ref{CL}, while $\xi_2$ is as in Proposition~\ref{prop:half-v-absorbers}. Further, let $C$ be as in Lemma~\ref{CL}, $D > 0$ -- as in Proposition~\ref{prop:absorbers}, and set
\begin{equation}\label{4in}
\gamma=\min\left\{ \left(\frac{\beta'}{24D}\right)^{2/3},\;\frac{\xi}{2(k+1)},\; \left(\frac{\beta'}{40}\right)^2,\; \frac1{(6\l(k+1)2^{k+3})^{2}}\right\}.
\end{equation}
Finally, fix any subset $R\subset V(G)$ of size $|R|\leq 2\gamma^2 n$.

In view of the discussion at the beginning of the section, it suffices to build an $m$-path containing at least $3\gamma^2n$ $(\xi,v)$-absorbers for every $v\in V$.

Let $\cA'$ be as in Proposition~\ref{A'}. Upon removing from $\cA'$ one $\xi$-absorber from each  intersecting pair, as well as all $\xi$-absorbers containing vertices from $R$, we obtain  a family $\cA$  which satisfies the following three conditions:
\begin{enumerate}[label=\rmlabel]
\item $|\cA| \leq 6\gamma^{3/2}n$;
\item all $\xi$-absorbers in $\cA$ are pairwise vertex disjoint;
\item for every $v\in V(G)$, there are at least $\frac18 \beta'\gamma^{3/2}  n-2\gamma^2n\ge 3\gamma^2 n$ $(\xi,v)$-absorbers in $\cA'$, where the last estimate follows from (\ref{4in}) and the fact that the absorbers in $\cA'$ are disjoint.
\end{enumerate}
There is a routine way to create the desired absorbing $m$-path from $\cA$ via repeated applications of Lemma~\ref{CL}.
Using Lemma~\ref{CL}, we  connect the  $\xi$-absorbers in $\cA$, one by one, into one  $m$-path $A$. Since each two consecutive $\xi$-absorbers on the $m$-path $A$ are connected by a sub-$m$-path with $\l(k+1)2^{k+1}$ internal vertices, by (\ref{4in}) and the inequality $m\le \l(k+1)$,
\[
|V(A)| \leq |\cA| \cdot (2m+ \l(k+1)2^{k+1}) \leq 6\gamma^{3/2}n \l(k+1)2^{k+2}\leq \frac{\gamma n}{2},
\]
as required.
In each step of the application of Lemma~\ref{CL}, the set $Z$ of forbidden vertices consists of the initial set $R$, the vertices in the $\xi$-absorbers in $\cA$, and the vertices used for the connections so far. Hence, by (\ref{4in}), even in the last step
\[
|Z| \leq |R|  + |V(A)| \leq 2\gamma^2 n + \frac{\gamma n}{2} \leq \gamma n \leq \frac{\xi n}{2(k+1)},
\]
which legitimates repeated applications of Lemma~\ref{CL}. Note that the $m$-ends of the obtained $m$-path $A$ are $\xi$-connectable, that is, $A$ is $\xi$-connectable. Moreover,  as stated in (iii), for every vertex $v\in V(G)$ the $m$-path $A$ contains at least $ 3\gamma^2 n$ (disjoint) $(\xi,v)$-absorbers. Consequently, for any set of vertices $U$ of size $|U| \leq 3\gamma^2 n$, one can absorb all the vertices from $U$ into $A$ obtaining an $m$-path $A_U$ with the same ends as $A$.
\end{proof}

\section{Covering Lemma}\label{sec:cover}

Our approach is similar to the one in the proof of Proposition 2.4 in \cite{DRRS}. The main new difficulty is to secure  $\xi$-connectable ends of the constructed $m$-paths. Here is an  outline of the proof.

We work under the hierarchy of constants
\begin{equation}\label{hiera}
 \eps \gg \gamma,\xi \gg T_0^{-1}, M^{-1}, \delta \gg T_1^{-1}, \gg  \tau \gg C^{-1}.
\end{equation}

In the first step we will take a $\delta$-quasirandom partition of the graph $G-Q$ and show that the associated reduced graph $\Gamma$ has a $C^k_{2k+2}$-factor (Claim \ref{8.2} below).
Then we will show that a.a.s.\ the subgraph of $H=G\cup G(n,p)$ corresponding to any copy of $C^k_{2k+2}$ in $\Gamma$ can be almost covered by not too many  vertex-disjoint $\xi$-connectable  $m$-paths (Claim \ref{8.3} below). The union of all these $m$-paths taken over all copies of $C^k_{2k+2}$ in a $C^k_{2k+2}$-factor of $\Gamma$ will constitute the desired family of $m$-paths.

We begin with the deterministic part. Consider a $\delta$-quasirandom partition
\[
V \setminus Q = V_0 \cup V_1 \cup \ldots \cup V_{T}
\]
of $G - Q$. Let $\Gamma$ be the reduced graph with respect to the above partition, namely, the vertex set of $\Gamma$ is $[T]$ and, for $1\le i< j\le T$, we include $\{i,j\}$ into $E(\Gamma)$ whenever $(V_i, V_j)$ is a $\delta$-quasirandom pair with density $d_{ij} = e(V_i,V_j)/|V_i||V_j| \geq \eps/3$.

\begin{claim}\label{8.2}
For all $\eps,\gamma,\delta>0$ with $\gamma+\delta\le \eps/6$, there is $T_0$ such that for all $T\ge T_0$, there exists a $C^k_{2k+2}$-factor $\cK$ in $\Gamma$.
\end{claim}
\proof
Take any $\eps,\gamma,\delta>0$ with $\gamma+\delta\le \eps/6$. Via Theorem \ref{AY} with $\eps:=\eps/3$ and $F:=C^k_{2k+2}$, choose $T_0$ and let $T\ge T_0$.
We first show that
\[
\delta(\Gamma)\ge\left(\frac{k}{k+1} + \frac{\eps}3 \right) T.
\]
Let us extend notation $e(U,W)$ to  intersecting sets $U$ and $W$ by counting twice the edges contained in $U \cap W$. In particular, for any $i=1,\dots,T$, $ e(V_i,V)=\sum_{v\in V_i}deg_G(v)$.
Thus, using the minimum degree condition imposed on $G$,
\[
e(V_i,V) \geq \left(\frac{k}{k+1} + \eps\right) |V_i| n.
\]
On the other hand, using the bound $|V_i|\le n/T$, the $\delta$-quasirandomness of the partition gives that
\begin{align*}
e(V_i,V) & \leq e(V_i,Q\cup V_0) + deg_{\Gamma}(i)|V_i|^2 + (T-deg_{\Gamma}(i))\left(\frac{\eps}3 + \delta\right)|V_i|^2 \\
& \leq (\gamma + \delta)|V_i|n + \frac{deg_{\Gamma}(i)}{T}|V_i|n + \left(\frac{\eps}3 + \delta\right) |V_i|n.
\end{align*}
Combining these two estimates and assuming that $\gamma + \delta \leq \eps/6$, we obtain, for all $i=1,\dots,T$, the lower bound
\[
deg_{\Gamma}(i) \geq \left(\frac{k}{k+1} + \frac{\eps}{3}\right) T.
\]
It is easy to check that $\chi(C^k_{2k+2})=k+1$. Hence, the existence of a $C^k_{2k+2}$-factor $\cK$ in $\Gamma$  follows by Theorem \ref{AY} applied with $\eps:=\eps/3$, $F:=C^k_{2k+2}$ and $G:=\Gamma$, for sufficiently large~$T$. \qed

\medskip

Turning to the union $H=G\cup G(n,p)$, we  now describe an event $\mathcal E=\cE(\xi, M,\tau)$ and show that it holds for  the random graph $G(n,p)$ a.a.s.
Fix a sequence $\seq{X} = (X_1,\ldots,X_{2k+2})$ of disjoint subsets of $V(G)$ and define a family $\cF (\seq{X})$  of copies of the graph
\[
B:=(k+1) B(\l,r,2M)
\]
as follows. Suppose that there is  a copy of the $\l M$-blow-up  $C^k_{2k+2}(\l M)$ in $G$ with each vertex class  $U_i \subseteq X_i$, $i=1,\ldots,2k+2$.
Then, we include in  $\cF (\seq{X})$ a copy $\tB$ of $B$ which is given by decomposition (\ref{C}) of Proposition~\ref{m-path decomposition} with $t:=2M$, \emph{provided} that the ends of the resulting $m$-path $P^m_{(2k+2)\l M}$ are $\xi$-connectable.

For any $\tau>0$, let
\[
\ccJ = \{ \seq{X} = (X_1,\ldots,X_{2k+2}) : |\cF (\seq{X})| \geq \tau n^{(2k+2)\l M}\}.
\]
The event $\cE$ holds if for every $\seq{X}\in\ccJ$ there is a subgraph $\tB \in \cF(\seq{X})$ with $\tB\subseteq G(n,p)$.

\begin{claim}\label{E} For every $\xi,M,\tau>0$,  there is $C\ge1$ such that for $p\ge Cn^{-2/\l}$ the event $\cE$ holds a.a.s.
\end{claim}
\proof
Let $c=c_{\tB}>0$ be a constant resulting from Proposition~\ref{Janson} with $t:=2M$, $F := \tB$, $\cF:=\cF(\seq{X})$. Further, let $C\ge (2k+3)/(c\tau^2)$. Suppose that $\ccJ \neq\emptyset$, since otherwise $\cE$ holds vacuously. For a given $\seq{X}\in\ccJ$, let $Y$ be the number of $\tB \in \cF(\seq{X})$ with $\tB\subseteq G(n,p)$. Then, by Proposition~\ref{Janson}
\[
\PP(Y=0)\le\PP(Y\le\tau\Psi_{\tB}/2)\le \exp\{-\tau^2cCn\}=o(2^{-(2k+2)n}).
\]
Since $|\ccJ| \leq 2^{(2k+2)n}$, by the union bound,
\[
\PP(\neg \cE)\le2^{(2k+2)n} \times o(2^{-(2k+2)n})=o(1).
\]
\qed

At the heart of the proof of Lemma~\ref{prop:covering} lies the following claim.

\begin{claim}\label{8.3}
For all $\eps>0$, $T$, and $M$, there exists $\xi>0$, $\gamma>0$, and $\delta>0$ such that for $C=C(\eps,M,\gamma)$ the following holds. If $\Gamma$ is the reduced graph of a $\delta$-quasirandom partition of $G-Q$ defined above and $K\subseteq V(\Gamma)$, $|K|=2k+2$, induces a copy of $C^k_{2k+2}$ in $\Gamma$,  then, with $V_K = \bigcup_{i\in K} V_i$, a.a.s.\ all but at most $\frac12\gamma^2|V_K|$ vertices of $H[V_K]$ can be covered by vertex disjoint $\xi$-connectable $m$-paths on $(2k+2)\l M$ vertices.
\end{claim}

\proof
Given $\eps>0$, let $\gamma,\delta$ be as in Claim \ref{8.2}, i.e. $\gamma + \delta \le \eps/6$. In addition, let
\begin{equation}\label{del}
\delta\le \frac{\gamma^4\left(\frac{\eps}{3}\right)^{e_F}}{16e_F},
\end{equation}
and let $M$ and $T$ be arbitrary.
Without loss of generality assume $K = [2k+2]$. Let $\cP$ be a largest collection of vertex-disjoint $\xi$-connectable $m$-paths in $H[V_K]$, each on $(2k+2)\l M$ vertices, with $\l M$ vertices in every $V_i$, $i=1,\dots,2k+2$. Let $X_i\subseteq V_i$, $i\in[2k+2]$, be the subset of $V_i$ consisting of all vertices not appearing on the $m$-paths in $\cP$. We have
\[
|X_1| = \ldots = |X_{2k+2}| = x
\]
for some integer $x$. It suffices to prove that $x \leq \frac12\gamma^2|V_i|$.

Assume that this is not the case. We will show that $\seq{X}=(X_1,\dots, X_{2k+2}) \in \ccJ$, which will further imply, using property $\cE$, the existence of a~$\xi$-connectable $(2k+2)\l M$-vertex $m$-path with vertex set contained in $X_1 \cup \ldots \cup X_{2k+2}$, thus contradicting the maximality of $\cP$.

Since $K=[2k+2]$ induces a copy of $C^k_{2k+2}$ in $\Gamma$, each pair $(V_i,V_{j})$, with $i,j$ lying at distance at most $k$ on $C^k_{2k+2}$, is  $\delta$-quasirandom in $G$ with density $d_{ij} \geq \eps/3$.

Let $\Omega'$ be the family of copies of $C^k_{2k+2}$ in $G[\seq{X}]$ with the property that each vertex  is contained in distinct $X_i$, $i\in [2k+2]$. By Fact \ref{heredit}, with $\alpha=\beta=\gamma^2/2$, the induced subgraph $G[\seq{X}]$  is $\delta'$-quasirandom with $\delta'=\tfrac{8\delta}{\gamma^4}$.
By Lemma~\ref{counting_lemma} applied to $F:=C^k_{2k+2}$ and $G:=G[\seq{X}]$, it follows, using also \eqref{del}, that
\[
|\Omega'| \geq \left( \left(\frac{\eps}{3}\right)^{e_F} - e_F\delta'\right) x^{v_F}\ge\frac{\left(\frac{\eps}{3}\right)^{e_F}}{2v_F^{v_F}}(v_Fx)^{v_F},
\]
where $e_F = (2k+2)k$ and $v_F = 2k+2$.
We are about to apply Proposition \ref{st} with $s := \l$. Let $t(\eps,\l)$ and $\xi$ be the resulting constants. Set $t=\max\{t(\eps,\l),\l M\}$. First we need to generate many copies of the $t$-blow-up of $C^k_{2k+2}$. By Corollary \ref{blow} with $q:=t$,
\[
d:=\frac{\left(\frac{\eps}{3}\right)^{e_F}}{2v_F^{v_F}},
\]
$F:=C^k_{2k+2}$, $G:=G[\seq{X}]$, and $\cF:=\Omega'$,
there are, for some $\tau'>0$, at least $\tau' (v_Fx)^{tv_F}$ copies of $C^k_{2k+2}(t)$ with each vertex class contained in distinct $X_i$, $i\in [2k+2]$. Let $\Omega''$ be the family of all these copies. In particular, $|\Omega''|\ge \tau'(v_Fx)^{tv_F}$.

Fix one member of $\Omega''$ with vertex classes $Y_1,\dots, Y_{2k+2}$ and apply Proposition~\ref{st} with $s:=\l $ twice,  to $Y_1,\dots,Y_{k+1}$ and to $Y_{k+2},\dots, Y_{2k+2}$. This way we find in $C^k_{2k+2}(t)$ a copy of $C^k_{2k+2}(\l )$ with vertex classes $W_1,\ldots,W_{2k+2}$, $W_i\subset Y_i\subset X_i$, $i=1,\ldots,2k+2$,
and such that the following property holds. For any $m$-tuple $\seq{x}=(x_1,\ldots,x_m)$ with $\{x_1,\ldots,x_r\} \subseteq W_{k+1}$, $\{x_{r+1},\ldots,x_{r+\l}\}=W_k$, \ldots,$\{x_{m-\l+1},\ldots,x_m\}=W_1$ and for any $m$-tuple $\seq{x}'=(x_1',\ldots,x_m')$ with $\{x_1',\ldots,x_r'\}\subseteq W_{k+2}$, $\{x_{r+1}',\ldots,x_{r+\l}'\}=W_{k+3}$, $\ldots$, $\{x_{m-\l+1}',\ldots,x_m'\}=W_{2k+2}$, both $\seq{x}$ and $\seq{x}'$ are $\xi$-connectable.

Let us extend arbitrarily this copy of $C^k_{2k+2}(\l )$ to a copy of $C^k_{2k+2}(\l M)$ with vertex classes $U_i$ such that $W_i\subset U_i\subset Y_i$, $i=1,\dots,2k+2$ (this is possible as $t\ge \ell M$). We order its vertices so that the associated copy of $P^m_{(2k+2)\l M}$ (see decomposition (\ref{C})) begins with $W_{1},\dots,W_{k+1}$ and ends with $W_{k+2},\dots, W_{2k+2}$, so that its ends are $\xi$-connectable.

Let $\Omega'''$ denote the family of all copies of $C^k_{2k+2}(\l M)$ in $G[\seq{X}$] as above.
We just showed that every member  of $\Omega''$ gives rise to at least one member of $\Omega'''$. On the other hand, each member of $\Omega'''$ can be obtained from at most $x^{v_F(t-\l M)}$ members of $\Omega''$. Thus, using the bound $|V_i|\ge(1-\delta)n/T\ge n/(2T)$, the assumption $x\ge \frac12 \gamma^2|V_i|$, and setting
\begin{equation}\label{ta�}
\tau_T:=\tau' v_F^{tv_F}(\gamma^2/{4T})^{v_F\l M},
\end{equation}
we have
\[
|\Omega'''|\ge \frac{\tau' (v_Fx)^{tv_F} } {x^{v_F(t-\l M)}} = \tau' v_F^{tv_F} x^{v_F\l M} \ge\tau_T n^{v_F\l M},
\]
and so $\seq{X} \in \ccJ$. Let $C=C(\tau_T)$ be as in Claim~\ref{E}. Then, a.a.s.\ the property $\cE$ holds, meaning that there is at least one copy of $B$ in $G(n,p)$ which, together with a copy of $C^k_{2k+2}(\l M)$ from $\Omega'''$, induces a $\xi$-connectable $(2k+2)\l M$-vertex $m$-path in $X_1\cup\ldots\cup X_{2k+2}$. \qed

\medskip

\begin{proof}[Proof of Lemma~\ref{prop:covering}] We begin by choosing constants as required implicitly by the preceding claims.
Given $\eps$, let $\gamma\le\eps/12$ and $M$ be so large that
\begin{equation}\label{M}
 ((2k+2)\l M)^{-1}\le \gamma^3 .
\end{equation}
Further, let  $\xi=\xi(\eps,\l)$ be as in Proposition~\ref{st}.
Next, choose an integer
\[
T_0=\max\left\{T_0(\eps/3,C^k_{2k+2}),\frac{4(2k+1)}{\gamma^2}\right\},
\]
where $T_0(\eps/3,C^k_{2k+2})$ is as in Theorem~\ref{AY}, and a constant $\delta > 0$ with
\[
\delta\leq \gamma^2/4
\]
satisfying (\ref{del}). Let $T_1=T_1(\delta,T_0)$ be given by Lemma \ref{regularity_lemma}. Finally, take $\tau=\tau_{T_1}$ as in (\ref{ta�}) and  $C=C(\tau)$ as in Claim~\ref{E}.

Apply the Szemer\'edi Regularity Lemma (Lemma~\ref{regularity_lemma}) to $H\setminus Q$ with $\delta$ and $T_0$ to obtain a partition $V\setminus Q = V_0 \dcup V_1 \dcup \ldots \dcup V_{T}$, with $T_0\le T\le T_1$. Let $\Gamma$ be the reduced graph with respect to that partition. By Claim~\ref{8.2}  there exists a $C^k_{2k+2}$-factor $\cK$ covering all but at most $2k+1$ vertices of $\Gamma$.
Applying Claim~\ref{8.3} to each $C^k_{2k+2}$ in $\cK$, we obtain a global family $\cP$ of vertex disjoint $\xi$-connectable $m$-paths in $H\setminus Q$, each having exactly $(2k+2)\l M$ vertices, covering all but at most
\[
\left(\delta + \frac{2k+1}{T_0} + \frac12\gamma^2\right)n\le \gamma^2 n
\]
vertices of $V\setminus Q$ (Here we use our assumptions on $\delta$ and $T_0$).  Moreover, the number of the paths in $\cP$ can be bounded from above by $\frac{n}{(2k+2)\l M}$ which, by (\ref{M}),  is at most $\gamma^3 n$.
\end{proof}


\section{Concluding remarks}\label{sec:con_rem}

Recall that the first case not covered by Corollary~\ref{cor:main_bounds} is $k=1$ and $m=5$.
We will see below that in this case the threshold, as defined in Definition~\ref{dirac_thr}, does not exist. The reason is that the range of $p=p(n)$ depends on $\alpha$ not only through the constant $C$ but also through the exponent of $n$. We believe that in many other cases the same is true as well.
First, let us focus on the lower bound. For convenience, we switch from $\alpha$ to $\eps=\alpha-\tfrac12$.

\begin{claim}\label{claim:15}
For each $0<\eps<1/9$ there exists a constant $c'=c'(\eps)>0$ and a sequence of $n$-vertex graphs $G_\eps:=G_\eps(n)$ such that $\delta(G_\eps)\ge \left(\tfrac12+\eps\right)n$ and for all $p\le  n^{-1/2-c'}$
\[
	\lim_{n\to\infty}\PP\big(G_\eps\cup G(n,p(n))\in\cC_n^5\big)=0\,.
\]
\end{claim}
\begin{proof}
Fix $0<\eps<1/9$ and define $c'=\frac{9\eps}{2-18\eps}$. Let $p=o(n^{-1/2-c'})$. Since $p=o(n^{-1/2})$, by Markov's inequality the number of copies of $K_4$ in $G(n,p)$ is a.a.s.\ $o(n)$. Now consider the graph $G_\eps:=G_{1/2+\eps}$ as described in the proof of Theorem~\ref{thm:main_bounds}. Assume that $H=G_\eps \cup G(n,p)$ contains a copy $C$ of $C_n^{5}$. After removing from $H$ all vertices in $W_1\cup W_2$ as well as at least one vertex of each copy of $K_4$ from $G(n,p)$ we obtain a subgraph $H' \subset H$ on $n-2\eps n - o(n)$ vertices. Observe that $H' \cap C$ contains a 5-path~$P$ of length
$n /  (2\eps n + o(n)) \ge \frac{1}{3\eps}$.
As in the proof of Theorem~\ref{thm:main_bounds} one can show that $G(n,p) \cap P$ contains a 2-path $Q$ on $q \ge \frac{1}{6\eps}$ vertices.
Observe that $Q$ has exactly $2q-3$ edges. Since
\[
\frac{q}{2q-3} = \frac{1}{2} + \frac{3}{4q-6} \le \frac{1}{2} + c',
\]
we have $p=o(n^{-q/(2q-3)})$ and, hence, Markov's inequality yields that a.a.s. $G(n,p)$ contains no 2-path on $q$ vertices, a contradiction.
\end{proof}

For the upper bound  it only follows from Theorem~\ref{thm:main} applied with $k=1$, $\ell=4$ and $r=1$ that the threshold is $O(n^{-1/2})$. It turns out that representing $m=5$ differently ($k=1$, $\ell=3$ and $r=2$) and taking a similar approach as in the proof of Theorem~\ref{thm:main} one can show  a better bound.
\begin{claim}\label{claim:15u}
For each $\eps>0$ there exists a constant $c''=c''(\eps)>0$ such that   for all $p\ge  n^{-1/2-c''}$
 \[
	\lim_{n\to\infty}\min_G\PP\big(G\cup G(n,p(n))\in\cC_n^5\big)=1\,,
\]
where the minimum is taken over all $n$-vertex graphs~$G$ with $\delta(G)\geq\left(\tfrac12+\eps\right) n$.
\end{claim}
\proof[Proof of Claim~\ref{claim:15u} (outline)] The key to the improvement of the bound on $p=p(n)$ is  a reformulation of Proposition \ref{Janson}
which yields the same bound $\PP(X \le\tau\Psi_F/2)\le \exp\{-\Omega(n)\}$ under milder assumptions on $p$. This is because now $m_B=d_{B(3,2,t)}=\tfrac{6t-3}{3t-1}<2$. In fact, it is true in more generality that for $\l\ge2$ and $r=\l-1$ (in which case $B(\l,r,t)=P^r_v$, $v=t\l$), for any $H\subseteq B(\l,\l-1,t)=P^r_v$, setting $v'=v_H$,
$$d_H\le d_{P^{\l-1}_{v'}}=\frac{v'(\l-1)-\binom\l2}{v'-1}\le \frac{v(\l-1)-\binom\l2}{v-1}=d_{P^{\l-1}_{v}}.$$
This means that taking $p=Cn^{-1/m_B}=Cn^{-1/2+1/(6(2t-1))}$ for sufficiently large $t$ (in fact, one should take $t=2M$, where $M$ is defined in Section \ref{sec:cover}) one can repeat every step of the proof of Theorem~\ref{thm:main}.
\qed

Determining the exact ``threshold'' for $\cC_n^5$ is left for the future work.

\medskip

Finally, let us emphasize that throughout this paper we have always assumed that $k\ge 1$.
There are some analogous results when $k=0$, i.e., assuming only that the minimum degree is a small fraction of $n$. It is known~\cite{BFM2003} that
$d_{0,1}(n) = n^{-1}$ and, in general, as it was shown in~\cite{BMPP2020} that $d_{0,m} = o(n^{-1/m})$ for $m\ge 2$. Determining the exact ``threshold'' for $m\ge 2$ is still open.


\subsection*{Acknowledgments} 
We are very grateful to both referees for their valuable remarks which have led to a better presentation of our results.


\begin{bibdiv}
\begin{biblist}


\bib{AY1996}{article}{
   author={Alon, Noga},
   author={Yuster, Raphael},
   title={$H$-factors in dense graphs},
   journal={J. Combin. Theory Ser. B},
   volume={66},
   date={1996},
   number={2},
   pages={269--282},
   issn={0095-8956},
   review={\MR{1376050}},
   doi={10.1006/jctb.1996.0020},
}

%
%

\bib{BFM2003}{article}{
   author={Bohman, Tom},
   author={Frieze, Alan},
   author={Martin, Ryan},
   title={How many random edges make a dense graph Hamiltonian?},
   journal={Random Structures Algorithms},
   volume={22},
   date={2003},
   number={1},
   pages={33--42},
   issn={1042-9832},
   review={\MR{1943857}},
}

\bib{BMPP2020}{article}{
   author={B\"{o}ttcher, Julia},
   author={Montgomery, Richard},
   author={Parczyk, Olaf},
   author={Person, Yury},
   title={Embedding spanning bounded degree graphs in randomly perturbed
   graphs},
   journal={Mathematika},
   volume={66},
   date={2020},
   number={2},
   pages={422--447},
   issn={0025-5793},
   review={\MR{4130332}},
   doi={10.1112/mtk.12005},
}

\bib{D1952}{article}{
   author={Dirac, G. A.},
   title={Some theorems on abstract graphs},
   journal={Proc. London Math. Soc. (3)},
   volume={2},
   date={1952},
   pages={69--81},
   issn={0024-6115},
   review={\MR{0047308}},
}

\bib{DRRS}{article}{
	author={Dudek, A.},
	author={Reiher, Chr.},
	author={Ruci\'nski, A.},
	author={Schacht, M.},
	title={Powers of Hamiltonian cycles in randomly augmented graphs},
        journal={Random Structures Algorithms},
        volume={56},
        number={1},
        date={2020},
        pages={122--141},
}

\bib{E1964}{article}{
   author={Erd\H{o}s, P.},
   title={On extremal problems of graphs and generalized graphs},
   journal={Israel J. Math.},
   volume={2},
   date={1964},
   pages={183--190},
   issn={0021-2172},
   review={\MR{0183654}},
   doi={10.1007/BF02759942},
}


\bib{ES1983}{article}{
   author={Erd\H{o}s, P.},
   author={Simonovits, M.},
   title={Supersaturated graphs and hypergraphs},
   journal={Combinatorica},
   volume={3},
   date={1983},
pages={181--192},
}

\bib{JLR}{book}{
   author={Janson, Svante},
   author={\L uczak, Tomasz},
   author={Ruci\'nski, Andrzej},
   title={Random graphs},
   series={Wiley-Interscience Series in Discrete Mathematics and
   Optimization},
   publisher={Wiley-Interscience, New York},
   date={2000},
   pages={xii+333},
   isbn={0-471-17541-2},
   review={\MR{1782847}},
}




\bib{KSS1998}{article}{
   author={Koml\'{o}s, J\'{a}nos},
   author={S\'{a}rk\"{o}zy, G\'{a}bor N.},
   author={Szemer\'{e}di, Endre},
   title={On the P\'{o}sa-Seymour conjecture},
   journal={J. Graph Theory},
   volume={29},
   date={1998},
   number={3},
   pages={167--176},
   issn={0364-9024},
   review={\MR{1647806}},
   doi={10.1002/(SICI)1097-0118(199811)29:3<167::AID-JGT4>3.0.CO;2-O},
}
\bib{NT}{article}{
author={Nenadov, Rajko},
   author={Truji\'c, Milos},
   title={Sprinkling a few random edges doubles the power},
   note={To appear in SIAM Journal on Discrete Mathematics}
   }

	
\bib{RRS2007}{article}{
   author={R\"{o}dl, Vojt\v{e}ch},
   author={Ruci\'{n}ski, Andrzej},
   author={Schacht, Mathias},
   title={Ramsey properties of random $k$-partite, $k$-uniform hypergraphs},
   journal={SIAM J. Discrete Math.},
   volume={21},
   date={2007},
   number={2},
   pages={442--460},
   issn={0895-4801},
   review={\MR{2318677}},
   doi={10.1137/060657492},
}

\bib{S1974}{article}{
author={Seymour, Paul D.},
title={Problem Section, Problem 3},
conference={
title={Combinatorics},
address={Proc. British Combinatorial Conf., Univ. Coll. Wales,
Aberystwyth},
date={1973},
},
book={
publisher={Cambridge Univ. Press, London},
},
date={1974},
pages={201--202. London Math. Soc. Lecture Note Ser., No. 13},
review={\MR{0345829}},
}

\bib{Sz1978}{article}{
author={Szemer\'edi, E.},
title={Regular partitions of graphs},
conference={
title={Probl\'emes combinatoires et th\'eorie des graphes},
address={Colloq.
Internat. CNRS, Univ. Orsay, Orsay},
date={1976},
},
book={
title={Colloq. Internat. CNRS}
publisher={CNRS, Paris},
volume={260},
},
date={1978},
pages={399--401},
}


\end{biblist}
\end{bibdiv}
\end{document}